\documentclass[11pt, letter, reqno]{amsart}
\usepackage{graphicx,psfrag}
\usepackage[psamsfonts]{amssymb}
\usepackage{pdfsync}
\usepackage{amscd}
\usepackage{amsmath}
\usepackage{mathrsfs}
\usepackage{stmaryrd}
\usepackage[small,nohug,heads=littlevee]{diagrams}
\usepackage{comment}
\usepackage{xfrac}
\usepackage[margin=1in]{geometry}
\usepackage{comment}
\usepackage{enumitem}
\usepackage{stmaryrd}
\usepackage[lowtilde]{url}
\usepackage{hyperref}
\allowdisplaybreaks

\theoremstyle{plain}
    \newtheorem{theorem}{Theorem}[section]
    \newtheorem{lemma}[theorem]{Lemma}
    \newtheorem{corollary}[theorem]{Corollary}
    \newtheorem{proposition}[theorem]{Proposition}
    \newtheorem{conjecture}[theorem]{Conjecture}
    
\theoremstyle{definition}
    \newtheorem{definition}[theorem]{Definition}

    \newtheorem*{thank}{Acknowledgements}
\theoremstyle{remark}
    \newtheorem{remark}[theorem]{Remark}

\numberwithin{equation}{section}

\newcommand{\wt}[1]{\widetilde{#1}}
\newcommand{\EE}{\mathcal{E}}

\newcommand{\ZZ}{\mathbb{Z}}

\newcommand{\R}{\mathbb{R}}
\newcommand{\C}{\mathbb{C}}

\newcommand{\OO}{\mathcal{O}}
\newcommand{\II}{\mathscr{I}}
\newcommand{\FF}{\mathscr{F}}
\newcommand{\EEE}{\mathscr{E}}
\newcommand{\LL}{\mathcal{L}}
\newcommand{\JJ}{\mathcal{J}}
\newcommand{\KK}{\mathcal{K}}

\newcommand{\F}{\mathcal{F}}
\newcommand{\M}{\mathcal{M}}

\newcommand{\lhom}{\mathcal{H}om}

\newcommand{\Id}{\operatorname{Id}}

\newcommand{\ind}{\operatorname{ind}}

\newcommand{\Lie}{\operatorname{Lie}}
\newcommand{\Lieg}{\mathfrak{g}}
\newcommand{\Lier}{\mathfrak{r}_t}
\newcommand{\Liek}{\mathfrak{k}}
\newcommand{\Lieh}{\mathfrak{h}}
\newcommand{\Lien}{\mathfrak{n}}
\newcommand{\Lieb}{\mathfrak{b}}
\newcommand{\Liep}{\mathfrak{p}}

\newcommand{\Liec}{\mathfrak{c}}
\newcommand{\Lies}{\mathfrak{s}}
\newcommand{\Lieu}{\mathfrak{u}}
\newcommand{\Liel}{\mathfrak{l}}
\newcommand{\Liet}{\mathfrak{t}}
\newcommand{\Liea}{\mathfrak{a}}
\newcommand{\Liem}{\mathfrak{m}}

\newcommand{\LLieg}{\mathfrak{g}^\circ}
\newcommand{\LLier}{\mathfrak{r}^\circ_t}
\newcommand{\LLieb}{\mathfrak{b}^\circ}
\newcommand{\LLien}{\mathfrak{n}^\circ}
\newcommand{\LLiek}{\mathfrak{k}^\circ}

\newcommand{\LLieh}{\mathfrak{h}^\circ}
\newcommand{\LLiet}{\mathfrak{t}^\circ}
\newcommand{\LLiea}{\mathfrak{a}^\circ}

\newcommand{\ld}{\lambda}
\newcommand{\lt}{{\lambda_t}}
\newcommand{\lc}{{\lambda_+}}
\newcommand{\lnc}{{\lambda_-}}
\newcommand{\lncp}{{\lambda'_-}}
\newcommand{\tlnc}{{\widetilde{\lambda}_-}}

\newcommand{\tg}{\wt{\Lieg}}

\newcommand{\gk}{(\Lieg,K)}

\newcommand{\gkt}{(\Lieg_t, \KK)}
\newcommand{\rkt}{(\Lier,\KK)}

\newcommand{\UU}{\mathcal{U} \Lieg^\circ}
\newcommand{\tD}{\widetilde{\D}}

\newcommand{\Dh}{\mathcal{D}_\mathfrak{h}}

\newcommand{\DQ}{\D^{\lambda}_{X \gets Q}}
\newcommand{\DY}{\D^{\lambda}_{X \gets Y}}

\newcommand{\DQt}{\D_{X \gets Q}^{\lt}}

\newcommand{\Dl}{\D_\lambda}

\newcommand{\Dr}{\D_{\lt}(\Lier)}
\newcommand{\DrQ}{\D^{\lt}_{X \gets Q}(\Lier)}
\newcommand{\JQ}{\JJ^Q}

\newcommand{\Sl}{\mathcal{S}_\lnc}
\newcommand{\Spl}{\mathcal{S}^{K}_\lnc}

\newcommand{\pb}{^{-1}}
\newcommand{\Mod}{\mathscr{M}}
\newcommand{\Ug}{\mathcal{U}\Lieg}
\newcommand{\Uh}{\mathcal{U}\Lieh}
\newcommand{\Ugh}{\mathcal{U}_{\Lieh}(\Lieg)}
\newcommand{\Urh}{\mathcal{U}_{\Lieh}(\Lier)}
\newcommand{\Urlt}{\mathcal{U}_{\lt}(\Lier)}
\newcommand{\Zg}{{\mathcal{Z}\Lieg}}
\newcommand{\U}{\mathcal{U}}

\newcommand{\D}{\mathcal{D}}

\newcommand{\I}{\mathcal{I}}

\newcommand{\Ind}{\operatorname{Ind}}

\newcommand{\supp}{\operatorname{Supp}}

\newcommand{\iqt}{\I(\lt,Q,\phi)}
\newcommand{\tiqt}{\widetilde{\I}(\lt,Q,\phi)}
\newcommand{\MQ}{\M(Q,\phi)}
\newcommand{\tMQ}{\widetilde{\M}(Q,\phi)}
\newcommand{\ilq}{\I(\lambda,Q,\phi)}
\newcommand{\llq}{\LL(\lambda,Q,\phi)}
\newcommand{\TX}{T^*X_{\lnc}}
\newcommand{\TXK}{{T^*_K X_{\lnc}}}
\newcommand{\TXQ}{{T^*_Q X_{\lnc}}}
\newcommand{\TXQcl}{T^*_{\overline{Q}} X_{\lnc}}
\newcommand{\Ad}{\operatorname{Ad}}
\newcommand{\Ql}{Q_{\lnc}}
\newcommand{\Qcl}{\overline{Q}_{\lnc}}

\newcommand{\W}{\mathcal{W}}
\newcommand{\V}{\mathscr{V}}
\newcommand{\VQ}{\V (\lambda,Q,\phi)}
\newcommand{\coh}{\M_{coh}}
\newcommand{\cohK}{\M_{coh}^K}
\newcommand{\ti}{\tilde{i}}
\newcommand{\tl}{\tilde{l}}
\newcommand{\iqzero}{\I_0(\lambda,Q,\phi)}
\newcommand{\cl}{\overline}
\newcommand{\pl}{p}
\newcommand{\bpl}{\bar{p}}

\newcommand{\SLR}{SL(2, \R)}

\newcommand{\PP}{\mathbb{P}^1}

\newcommand{\GR}{G_\R}
\newcommand{\Gmot}{G_{\R,0}}
\newcommand{\KR}{K_\R}

\newcommand{\KC}{K}

\newcommand{\Llr}{L^\lnc_\R}

\newcommand{\Image}{\operatorname{Im}}

\begin{document}

\title{Mackey analogy as deformation of $\D$-modules}
\author{Shilin Yu}
\address{Room 712, Academic Building No.1,
The Chinese University of Hong Kong,
Shatin, N.T., Hong Kong
}
\email{slyu@math.cuhk.edu.hk}



%\date{\today}

%\abovedisplayskip=4pt
%\belowdisplayskip=4pt

\setlength{\abovedisplayskip}{7pt}
\setlength{\belowdisplayskip}{7pt}

\begin{abstract}
 Given a real reductive group Lie group $G_\mathbb{R}$, the Mackey analogy is a bijection between the set of irreducible tempered representations of $G_\mathbb{R}$ and the set of irreducible unitary representations of its Cartan motion group. We show that this bijection arises naturally from families of twisted $\mathcal{D}$-modules over the flag variety of $G_\mathbb{R}$. 
\end{abstract}

\maketitle

\tableofcontents

\section{Introduction}\label{sec:intro}

The aim of this paper is to give a natural conceptual explanation of an analogy between tempered representations of reductive Lie groups and that of their degenerations, which was first conjectured by Mackey (\cite{Mackey}). We show that this analogy can be understood from the perspective of families of representations. We obtain such families by considering deformation of $\D$-modules, which play a fundamental role in the connection between representation theory with algebraic geometry (\cite{BeilinsonBernstein}). The case of $\SLR$ has been examined in the paper \cite{YuMackeySL} of the same author with others. In this paper we will deal with arbitrary reductive Lie groups and verify the claimed results and conjecture in \emph{loc. cit}.

The idea of families of representations have long been studied in mathmatical physics under the name of \emph{contractions} of groups and their representations (\cite{Segal}, \cite{InonuWigner}). We will consider only one important case in this paper. Let $\GR$ be a noncompact reductive Lie group with a maximal compact subgroup $K_\R$. The  \emph{Cartan motion group} of $G_\R$ is defined to be the group
\[ G_{\R,0} :=K_\R \ltimes \Lies_\R, \]
where $\Lieg_\R = \Lie(G_\R)$ and $\Liek_\R = \Lie(K_\R)$ are the corresponding Lie algebras and $\Lies_\R \cong \Lieg_\R / \Liek_\R$ is regarded as an abelian group with the usual addition of vectors. In 1970s, Mackey suggested (\cite{Mackey}) that there might be a bijection between the irreducible tempered representations of a noncompact semisimple group $G_\R$ and the irreducible tempered representations of $G_{0,\R}$. It is counterintuitive at first sight since the algebraic structures of the groups $G_\R$ and $G_{\R,0}$ are quite different. The classification of tempered representations of $\GR$ is a nontrivial matter, which was due to the work of Knapp and Zuckerman (\cite{KnappZuckerman_I}, \cite{KnappZuckerman_II}). On the other hand, Mackey himself developed a full theory of representations of semidirect product groups like $G_{\R,0}$ (\cite{MackeyInd}), so the unitary dual of $G_{\R,0}$ is much easier to describe.

%Mackey himself was very cautious even after making a number of calculations in support of his conjecture. One reason is that, even if the analogy holds for parameter spaces, it behaves poorly at the level of representation spaces. For instance, while there are unitary unitary irreducible representations of $G_{\R,0}$ whose underlying vector spaces are of finite dimensions (on which the $\Lieg_\R / \Liek_\R$ part of $G_{\R,0}$ acts trivially), all nontrivial unitary representations of $G_\R$ are infinite-dimensional. 

One of the evidences of the Mackey analogy comes from the Connes-Kasparov isomorphism in K-theory of $C^*$-algebras (\cite{BCH}), which implies that the tempered duals of $\GR$ and $G_{\R,0}$ are equivalent $K$-theoretically (see also \cite{HigsonMackeyKtheory}). Higson conjectured that there actually exists a set theoretical bijection. In other words, Mackey analogy, if holds, can be regarded as a refined version of the Connes-Kasparov isomorphism. In his paper \cite{HigsonMackeyKtheory}, Higson examined the case where $G_\R$ is a connected complex semisimple group (regarded as a real group) and showed that there is a natural bijection between the reduced duals of $G_\R$ and $G_{\R,0}$. Later in \cite{HigsonMackey}, he strengthened this result by showing that there is even a natural bijection between the admissible duals of $G_\R$ and $G_{\R,0}$ when $G_\R$ is a complex group. Other work on examining special cases of the conjecture are \cite{George}, \cite{Skukalek}.

The conjectured bijection of Mackey and Higson has been recently established for general real reductive Lie group by Afgoustidis in \cite{Afgoustidis}. It will be summarized in Section \ref{subsec:Afg}. We will call it the \emph{Mackey-Higson-Afgoustidis (MHA) bijection} throughout the paper. Afgoustidis used the MHA bijection to give a new proof of the Connes-Kasparov isomorphism for real reductive Lie groups (\cite{Afgoustidis_CK}). Moreover, he also studied the Mackey analogy at the level of representation spaces and got partial results. However, a conceptual understanding of the MHA bijection was missing.

As proposed in our former work \cite{YuMackeySL}, Mackey analogy can be understood by constructing families of standard Harish-Chandra $\D$-modules. The theory of $\D$-modules invented by Beilinson and Bernstein (\cite{BeilinsonBernstein}) has built a bridge between algebraic geometry and representation theory. It was shown that standard Harish-Chandra $\D$-modules satisfying certain conditions (called tempered Harish-Chandra sheaves) give an alternative classification of irreducible tempered representations (\cite{Chang},\cite{Mirkovic}), parallel to Knapp-Zuckerman. We will review these results in Section 3. In a nutshell, from a given Cartan subalgebra $\Lieh_\R$ of $\Lieg_\R$ together with a decomposition $\Lieh_\R = \Liet_\R \oplus \Liea_\R$ determined by a Cartan involution one can construct a family of $\D$-modules depending on a discrete weight parameter in $\Liet^*_\R$ and a continuous parameter in $i\Liea^*_\R$ such that the global sections form a family of tempered representations of $\GR$. Here is how we are going to construct the one-parameter family related to Mackey analogy: given one of the $\D$-modules in the family, we rescale the parameter in $i\Liea^*_\R$ and let it go to infinity, so that the resulting one-parameter family of $\GR$-representations  `converges' to a representation of $G_{\R,0}$. The $G_{\R,0}$-representation might not be irreducible, but it has a unique maximal quotient, which is the one that appears in the MHA bijection. This explains the phenomena that underlying vector spaces of representations related by the MHA bijection can be of different `sizes'.

Special care needs to be taken, however, when the parameter in $i\Liea^*_\R$ is not generic. For instance, as shown in \cite{YuMackeySL}, in the case of the even principal series representation of $\SLR$ with zero infinitesimal character, the naive construction gives the sum of all even weights of $\KR=SO(2)$ regarded as $(\Lieg_0,K)$-module with $\Lies = \Lies_\R \otimes_\R \C \subset \Lieg_0 = \Lieg_{\R,0} \otimes_\R \C$ acting trivially, which is therefore not finitely-generated. In \cite{YuMackeySL}, this issue was fixed by taking the subfamily generated by the minimal $K$-types, which is a fundamental notion in the study of representations due to Vogan (\cite{VoganThesis}) and also plays crucial role in the MHA bijection. We will prove here that this construction works in general case. For this we need detailed properties of $K$-orbits on the flag variety and minimal $K$-types in terms of $\D$-modules, which we will review in \S~\ref{subsec:Korbit} and \S~\ref{subsec:minKtypes} respectively. The main result is Theorem \ref{thm:main}.

It is natural to wonder about the meaning of the families of $\D$-modules constructed here. One feature of our construction is that the `limit' of a family gives a $K$-equivariant coherent sheaf over some twisted cotangent bundle of the flag variety. Its support is the analogue of the characteristic variety of a $\D$-module, which always lie in the usual cotangent bundle. Thus we call it the \emph{twisted characteristic variety}. It preserves more information than the usual one by remembering part of the infinitesimal character of the $\D$-module we start with. Just as the usual characteristic variety, it can be shown that the smooth locus of the twisted characteristic variety is Lagrangian inside the twisted cotangent bundle equipped with its natural algebraic symplectic structure. We expect that the existence of such natural families can be interpreted as deformation quantization of Lagrangian subvarieties inside twisted cotangent bundles over the flag variety, in the sense of \cite{BGKDJ} (also see\cite{NestTsygan}, \cite{AgnoloSchapira}). This viewpoint will be pursued elsewhere in the future.

Recently a general theory of families of Harish-Chandra modules has been developed in \cite{BHS1}. In this framework, the authors of \emph{loc. cit.} studied contractions of representations of $\SLR$ in \cite{BHS2} and the MHA bijection for admissible dual of $\SLR$ in \cite{Subag} using the techniques of Jantzen filtration. We expect that there is a close connection between their approach and ours. 

\begin{thank}
  The author would like to thank Jonathan Block, Justin Hilburn, Nigel Higson and Tony Pantev for numerous discussions. The author appreciates Alexandre Afgoustidis for sharing the draft of his paper at early stage and for the detailed explanations of his work. The author also wants to express gratitude to Dragan Mili{\v{c}}i{\'c}, Wilfried Schmid and David Vogan for their great patience with the author's elementary questions about representation theory. The hospitality of Jeffrey Adams and David Vogan during the author's visit to University of Maryland and Massachusetts Institute of Technology respectively is gratefully acknowledged.
  
  Special thanks go to Junyan Cao, who offered the floor of his hotel room to the author during the `Algebraic Geometry 2015' conference at University of Utah, where the main idea of this paper came up.   
  
  The author was supported by the Direct Grants and Research Fellowship Scheme from The Chinese University of Hong Kong. 
\end{thank}

{\bf Notations.} For a real Lie group we always put the subscript $\R$ following the capital letter, e.g., $\GR$. Same for their Lie algebras and related subspaces, e.g., $\Lieg_\R$. For the complexified Lie groups and their Lie algebras, we drop the subscripts.

%Note that Vogan's theory of minimal $K$-types (\cite{VoganThesis}) is as indispensable in our construction as Afgoustidis' construction of the MHA bijection, which in particular generalizes Vogan's bijection between irreducible tempered representations with real infinitesimal character and their minimal $K$-types (see \S \ref{sec:Mackey}, especially Theorem \ref{thm:Vogan}).

%In particular, we regard representations as $\D$-modules over the flag variety via the Beilinson-Bernstein localization theorem (\cite{BeilinsonBernstein}) and show how to deform them to get representations of the Cartan motion group.  We expect that our construction works in general for admissible representations of $G_\R$. We will state our main conjecture in Conjecture \ref{conj:Mackey} at the end of \S~\ref{subsec:deform}. 

\section{The Mackey-Higson-Afgoustidis bijection}\label{sec:Mackey}

\subsection{Afgoustidis's Mackey-Higson bijection}\label{subsec:Afg}

We briefly describe Afgoustidis's Mackey-Higson bijection for tempered representations at the level of parameter spaces (\cite{Afgoustidis}). Suppose $G_\R$ is a connected real semisimple Lie group with finite center. Fix a maximal compact subgroup $K_\R$ of $G_\R$. Let $\theta_\R$ and $\theta=\theta_\C$ be the corresponding Cartan involutions of $\Lieg_\R$ and $\Lieg$ respectively, which have $\Liek_\R$ and $\Liek$ as fixed points respectively. Let $\Lies_\R$ and $\Lies=\Lies_\C$ be the $-1$-eigenspaces of the involutions $\theta_\R$ and $\theta$ respectively, so we have $\Lieg_\R=\Liek_\R \oplus \Lies_\R$ and $\Lieg=\Liek \oplus \Lies$. The motion group $G_{\R,0} = K_\R \ltimes \Lies_\R$ has Lie algebra $\Lieg_{\R,0} = \Liek_\R \ltimes \Lies_\R$, whose complexification is $\Lieg_0 = \Liek \ltimes \Lies$.

Suppose $\Liea_\R$ is a maximal abelian subalgebra of $\Lies_\R$. Let $W_\Liea$ be the Weyl group of the pair $(\Lieg_\R,\Liea_\R)$. We consider $K_\R$-orbits in $\Lies^*_\R$. Any  $K_\R$-orbit in $\Lies_\R^*$ intersects with $\Liea_\R^*$ at a unique $W_\Liea$-orbit, where we identify $\Liea_\R^*$ as a subspace of $\Lies^*_\R$ using the Killing form. Hence choosing a $K_\R$-orbit of $\Lies^*_\R$ is equivalent to choosing a character $\chi \in \Liea^*_\R$ up to a $W_\Liea$-symmetry. Denote the stabilizer of $\chi$ in $K_\R$ by $K_\R^\chi$. According to Afgoustidis, a \emph{Mackey datum} is a pair $(\chi, \sigma)$ in which $\sigma$ is an irreducible unitary representation of $K^\chi_\R$. Given such a datum, we can produce a unitary representation $M_0(\delta)$ of $G_{\R,0}$ by parabolic induction,
\begin{equation}\label{eq:M0}
 M_0(\delta) := \Ind^{\Gmot}_{K_\R^\chi \ltimes \Lies_\R} \left[ \sigma \otimes e^{i \chi} \right]. 
\end{equation}
Mackey showed that all such $M_0(\delta)$ give a complete list of irreducible unitary representations of $G_{\R,0}$ (see Section 7 of \cite{Mackey49} or Chapter 3 of \cite{Mackey76}). They are all tempered. Moreover, two Mackey data $\delta_1=(\chi_1,\sigma_1)$ and $\delta_2=(\chi_2,\sigma_2)$ give rise to unitarily equivalent representations if and only if there is an element of the Weyl group $W_\Liea$ which sends $\chi_1$ to $\chi_2$ and $\sigma_1$ to an irreducible $K_{\chi_2}$-representation which is unitarily equivalent with $\sigma_2$. We say that the Mackey data $\delta_1$ and $\delta_2$ are equivalent.

We now recall the definition of \emph{minimal/lowest $K$-type} due to Vogan (\cite{VoganThesis}). 

\begin{definition}\label{defn:minimal}
 Suppose $\sigma$ is an irreducible representation of $\KR$. By identifying $\sigma$ with its highest weight, we can define the norm
  \[ \lvert \sigma  \rvert :=  \langle \sigma + 2\rho_c, \sigma + 2\rho_c  \rangle,  \]
where $\rho_c$ is the half sum of positive roots of $\KR$. For a $(\Lieg,K)$-module $\pi$, $\sigma$ is called a \emph{$K$-type} of $\pi$ if $\sigma$ occurs in $\pi$ restricted to $\KR$, and is called  a \emph{minimal/lowest $K$-type} if furthermore $\lvert \sigma \rvert$ is minimal among all $K$-types.
\end{definition}

Afgoustidis' construction relies on the following crucial result by Vogan (\cite{Vogan}).

\begin{theorem}[Vogan, \cite{Vogan}]\label{thm:Vogan}
  Any irreducible tempered representation of a real reductive Lie group $G_\R$ with real infinitesimal character has a unique $K_\R$-type. This defines a bijection between the equivalence classes of irreducible tempered representations of $G_\R$ with real infinitesimal characters and the equivalence classes of irreducible representations of $K_\R$. 
\end{theorem}

Now fix a maximal torus $T_\R$ of $K_\R$.  Given a Mackey datum $\delta=(\chi,\sigma)$, regard $\chi$ as an element in $\Lieg^*_\R$ and consider the centralizer $L^\chi_\R$  of $\chi$ for the coadjoint action of $\GR$ on $\Lieg^*$. Then we have the Langlands decomposition $L^\chi_\R=M^\chi_\R A^\chi_\R$. Choose a system $R^+$ of positive roots for the pair $(\Lieg, \Lieh)$ and define the nilpotent subalgebra $\Lien^\chi_\R$ of $\Lieg_\R$ as the real part of the sum of root spaces in $\Lieg$ for those positive roots which do not vanish on $\Liea^\chi_\R$. Set $N^\chi_\R:=\exp_{\GR}(\Lien^\chi_\R)$ and $P^\chi_\R:=M^\chi_\R A^\chi_\R N^\chi_\R$. Then $P^\chi_\R$ is a cuspidal parabolic subgroup.

The $K^\chi_\R$-representation $\sigma$ in the Mackey datum $\delta$ determines a tempered representation $V_{M^\chi_\R}(\sigma)$ of $M^\chi_\R$ with real infinitesimal character by Theorem \ref{thm:Vogan}. The character $\chi$ determines a one-dimensional unitary representation $e^{i\chi}$ of $A^\chi_\R$. The work of Harish-Chandra, Knapp-Zuckerman (\cite{KnappZuckerman_I}, \cite{KnappZuckerman_II}) showed that the induced unitary representation
  \[ M(\delta) := \Ind_{P^\chi_\R}^{G_\R} \left[ V_{M_\R^\chi}(\sigma) \otimes e^{i \chi} \otimes 1  \right]  \] 
of $\GR$ is irreducible and tempered. Moreover, $M(\delta_1)$ and $M(\delta_2)$ are unitarily equivalent if and only if $\delta_1$ and $\delta_2$ are equivalent as Mackey data. Afgoustidis has proven that

\begin{theorem}[\cite{Afgoustidis}]\label{thm:afg}
The correspondence
  \[
    M(\delta)  \longleftrightarrow M_0(\delta)
  \]
is a bijection between the tempered dual of $G_\R$ and the unitary (tempered) dual of $G_{\R,0}$.
\end{theorem}

\begin{definition}
  The bijection in Theorem \ref{thm:afg} is called the \emph{Mackey-Higson-Afgoustidis (MHA) bijection}. 
\end{definition}

\section{Review of Harish-Chandra sheaves}\label{sec:Dmod}

\subsection{Basics of $\D$-modules}\label{subsec:Dmod}

Throughout the rest of the paper we fix a connected reductive algebraic group $G$ defined over $\R$ and fix a real Lie group $\GR$ which has finite index in the set of real points of $G$. This assumption assures that the component group of any Cartan subgroup of $\GR$ is finite and abelian. We fix a maximal compact subgroup $\KR$ of $\GR$ and write $K \subset G$ for its complexification. Any finite-dimensional $\KR$-module extends uniquely to an algebraic $K$-module, so we will not distinguish them. This is the setting used in \cite{Chang}. However, we remark that all the results in this paper, with slight modification in the statements, hold for more general $\GR$, at least in the setting of \cite{Mirkovic}. 

%For instance, the component groups of Cartan subgroups could be nonabelian, so that we need to use $K$-equivariant vector bundles instead of line bundles in the construction of standard Harish-Chandra sheaves (cf., \S \ref{subsec:HC}).

We recall the construction of twisted $\D$-modules on the flag variety following \cite{Milicic}. Let $X$ be the flag variety of $G$, which is the variety of all Borel subalgebras $\Lieb$ in $\Lieg$. Let $\LLieg=\OO_X \otimes_\C \Lieg$ be the sheaf of local sections of the trivial bundle $X \times \Lieg$. Let $\LLieb$ be the vector bundle on $X$ whose fiber $\Lieb_x$ at any point $x$ of $X$ is the Borel subalgebra $\Lieb \subset \Lieg$ corresponding to $x$. Similarly, let $\LLien$ be the vector bundle whose fiber $\Lien_x$ is the nilpotent ideal $\Lien_x = [\Lieb_x, \Lieb_x]$ of the corresponding Borel subalgebra $\Lieb_x$. $\LLieb$ and $\LLien$ can be considered as subsheaves of $\LLieg$. $\LLieg$ has a natural structure of Lie algebroid: the differential of the action of $G$ on $X$ defines a natural map from $\Lieg$ to sections of tangent bundle $TX$ of $X$ and hence induces an anchor map $\tau: \LLieg \to TX$. The Lie structure on $\LLieg$ is given by
  \[  [f \otimes \xi, g \otimes \eta] = f \tau(\xi) g \otimes \eta - g \tau(\eta)f \otimes \xi + fg \otimes [\xi, \eta] \]
for any local functions $f, g \in \OO_X$ and $\xi, \eta \in \Lieg$. The kernel of $\tau$ is exactly $\LLieb$, so $\LLieb$ and $\LLien$ are sheaves of Lie ideals in $\LLieg$.

We then form the universal enveloping algebra of the Lie algebroid $\LLieg$, which is the sheaf $\UU=\OO_X \otimes_\C \Ug$ of associative algebras with the multiplication defined by
  \[  (f \otimes \xi) (g \otimes \eta) = f \tau(\xi) g \otimes \eta + fg \otimes \xi \eta \]
for $f,g \in \OO_X$ and $\xi \in \Lieg$, $\eta \in \Ug$. The sheaf of left ideals $\UU \LLien $ generated by the sheaf of Lie ideals $\LLien$ in $\UU$ is a sheaf of two-sided ideals in $\UU$, hence the quotient $\Dh = \UU /  \UU \LLien$ is a sheaf of associative algebras on $X$.

The natural morphism from $\LLieg$ to $\Dh$ induces an inclusion of $\LLieh=\LLieb/\LLien$ into $\Dh$. The vector bundle $\LLieh$ turns out to be a trivial vector bundle and its global sections over $X$ is the \emph{abstract Cartan algebra} $\Lieh$ of $\Lieg$, which is independent of the choice of Borel subalgebra. Moreover, we also have abstract root system $\Sigma$ and positive root system $\Sigma^+$ in $\Lieh^*$ which consists of the set of roots of $\Lieh$ in $\Lieg/\Lieb_x$, as well as the abstract Weyl group $W$. For details, see Section 3.1, \cite{Ginzburg-book}. The natural action of $G$ on $\LLieh$ is trivial and embedding $\LLieh \hookrightarrow \Dh$ identifies the universal enveloping algebra $\Uh = S \Lieh$ of the abelian Lie algebra $\Lieh$ with the $G$-invariant part of $\Gamma(X,\Dh)$. Here $S$ denotes the symmetric algebra. On the other hand, the center $\Zg$ of $\Ug$ is also naturally contained in $\Gamma(X,\Dh)^G$ and the induced map $\gamma: \Zg \to S\Lieh$ is the well-known Harish-Chandra homomorphism, which identifies $\Zg$ with the $W$-invariant of $S\Lieh$, where the action of $W$ on $\Lieh^*$ is the usual one twisted by the half sum $\rho$ of positive roots:
  \[ w . \lambda = w (\lambda - \rho) + \rho. \]
In other words, we have the Harish-Chandra isomorphism
  \[ \gamma_{HC}: \Zg \xrightarrow{\sim} (S\Lieh)^{W, .}  \]
Define the algebra 
  \[ \U_\Lieh(\Lieg) := \U\Lieg \otimes_{\Zg} S \Lieh. \] 
Note that natural homomorphism $\U\Lieg \to \Ugh$
is injective since $\U\Lieg$ is a free module over $\Zg$ (\cite{Kostant}). The natural homomorphism $\U\Lieh \to \Ugh$ is also injective since $S\Lieh$ is free over $\Zg \simeq (S\Lieh)^{W,.}$. So we can regard  $\U\Lieg$ and $S\Lieh$ as subalgebras of $\U_\Lieh(\Lieg)$. Moreover, the center of $\U_\Lieh(\Lieg)$ is exactly $S\Lieh$.

\begin{proposition}[Lemma 3.1., \cite{Milicic}]\label{prop:Dh}
The natural morphism 
  \[ \Ugh \to \Gamma(X,\Dh)  \]
is an isomorphism of algebras. Moreover, $H^i(X, \Dh) = 0$ for $i > 0$.
\end{proposition}

On the other hand, for any Cartan subalgebra $\Liec$ of $\Lieg$ and any Borel subalgebra $\Lieb_x$ containing $\Liec$, the composition 
  \[ \Liec \to \Lieb_x \to \Lieb_x / \Lien_x \simeq \Gamma(X,\LLieh) = \Lieh \] 
is an isomorphism, which does depend on the choice of $x \in X$. Both this isomorphism and the dual isomorphism $\Lieh^* \to \Liec^*$ are called \emph{specialization} at $x$. For a fixed $\Liec$ and two different points in $X$, the resulting specialization isomorphisms differ by composition with an element of Weyl group.

Any $\lambda \in \Lieh^*$ determines a homomorphism from $\Uh$ to $\C$. Let $I_\lambda$ be the kernel of the homomorphism $\Uh \to \C$ determined by $\lambda - \rho$. Then $\gamma\pb(I_\lambda)$ is a maximal ideal in $\Zg$ and $\gamma\pb(I_\lambda) = \gamma\pb(I_\sigma)$ if and only if $w \cdot \lambda = \sigma$ for some $w \in W$ (here we use the usual $W$-action on $\Lieh$). Thus we can denote the kernel by $J_\chi = \gamma\pb(I_\lambda)$ where $\chi = W \cdot \lambda$ is the $W$-orbit of $\lambda$ in $\Lieh^*$. We denote the corresponding infinitesimal character by $\chi_\lambda: \Zg \to \C$. The sheaf $I_\lambda \Dh$ is a subsheaf of two-sided ideals in $\Dh$, therefore $\Dl = \Dh / I_\lambda \Dh$ is a sheaf of associative algebras. The elements of $J_\chi$ map into the zero section of $\Dl$. Therefore we have a canonical morphism of $\mathcal{U}_\chi := \Ug / J_\chi \Ug = \Ug \otimes_\Zg \C_{\lambda - \rho}$ into $\Gamma(X,\Dl)$.

\begin{proposition}[Theorem 3.2., \cite{Milicic}]\label{prop:Uchi}
The canonical morphism
  \[ \mathcal{U}_\chi \to \Gamma(X,\D_\lambda) \]
is an isomorphism of algebra. Moreover, $H^i(X, \Dl) = 0$ for $i > 0$. 
\end{proposition}

Let $\Mod(\mathcal{U}_\chi)$ be the abelian category of $\mathcal{U}_\chi$-modules. It is the same as the category of $\Ug$-modules with infinitesimal characters determined by $\chi$. We also have the abelian category $\Mod(\Dl)$ of quasi-coherent left $\Dl$-modules over $X$. The Beilinson-Bernstein localization theorem (\cite{BeilinsonBernstein}) is stated as follows. 

\begin{theorem}
The global section functor $\Gamma: \Mod(\Dl) \to \Mod(\mathcal{U}_\chi)$ is an equivalence of abelian categories if $\lambda$ is dominant and regular. The inverse functor is $\Delta_\lambda: \Mod(\mathcal{U}_\chi) \to \Mod(\Dl)$, given by $\Delta_\lambda(V) = \Dl \otimes_{\mathcal{U}_\chi} V$, for any $V \in \Mod(\mathcal{U}_\chi)$.
\end{theorem}

When $\lambda$ is regular but not dominant, the Beilinson-Bernstein localization theorem still holds if one replaces the abelian categories by their derived categories and the functors by their derived versions. When $\lambda$ is singular, the situation is more complicated since there can be nontrivial $\D$-modules with trivial sheaf cohomology groups. See \cite{Milicic} for details for example.

%Fortunately, in the case of standard tempered Harish-Chandra modules we still have nice classfication results (Theorem \ref{thm:Dtempered}).

The localization of Harish-Chandra modules of the pair $(\Lieg,K)$ are \emph{Harish-Chandra sheaves}, which are coherent $\Dl$-modules with $K$-equivariant structures which is compatible with the $\Dl$-module structure in the same as Harish-Chandra modules are defined. For a precise definition, see \cite{BernsteinLunts}. The Harish-Chandra sheaves form an abelian category $\Mod_{coh}(\Dl,K)$. Denote by $\Mod(\mathcal{U}_\chi,K)$ the category of Harish-Chandra modules whose $\Ug$-module structure factors through $\mathcal{U}_\chi$. Then the equivariant version of the Beilinson-Bernstein localization theorem reads as

\begin{theorem}
  Let $\lambda \in \Lieh^*$ be dominant and regular. Then the localization functor $\Gamma:  \Mod_{coh}(\Dl,K) \to \Mod(\mathcal{U}_\chi,K)$ is an equivalence of categories. Its inverse is $\Delta_\lambda$.
\end{theorem}

\subsection{Standard Harish-Chandra sheaves}\label{subsec:HC}

We recall the construction of standard Harish-Chandra sheaves. Let $Y$ be any smooth subvariety of $X$. Denote the embedding by $i: Y \hookrightarrow X$ and the ideal sheaf of $Y$ by $\II_Y$. 

\begin{definition}\label{defn:transfer_undeformed}
  \begin{enumerate}
    \item
      The \emph{sheaf of twisted differential operators over $Y$ induced by $i$ and $\lambda$} is
      \[ \D^i_\lambda := \{ A \in \Dl ~|~ A \cdot \II_Y \subset \II_Y \cdot \Dl \} / \II_Y \cdot \Dl.  \]
     \item
       The \emph{transfer bimodule for $\Dl$-modules} is 
         \[ \DY := i\pb\D_\lambda \otimes_{i\pb \OO_{X}} \omega_{Y/X},  \]
        where $\omega_{Y/X} = \omega_X\pb \otimes_{i\pb \OO_X} \omega_Y$ is the relative canonical bundle of $Y$ in $X$.
  \end{enumerate}
\end{definition}

\begin{lemma}[Claim 4.11., \cite{Chang}]
  The transfer bimodule $\DY$ is an $i\pb \Dl$-$\Dl^i$ bimodule.
\end{lemma}

Denote by $\Mod(\Dl^i)$ the category of sheaves of quasi-coherent left $\Dl^i$-modules. Define the direct image functor $i_+: \Mod(\Dl^i) \to \Mod(\Dl)$ by
  \[ i_+ \FF = i_*(\DQ \otimes_{\Dl^i} \FF).  \]

Recall the following result due to Matsuki \cite{Matsuki1}. 

\begin{lemma}\label{lemma:KCartan}
  Let $\Lieb$ be a Borel subalgebra of $\Lieg$ and $N$ the unipotent radical of the Borel subgroup $B$ of $G=\mathrm{Int}(\Lieg)$ corresponding to $\Lieb$. Then $\Lieb$ contains a $\theta$-stable Cartan subalgebra $\Liec$. Morevoer, all $\theta$-stable Cartan subalgebras of $\Lieb$ are conjugate by $K \cap N$.
\end{lemma}

Now let $Y$ be a $K$-orbit $Q$ in $X$. Then Lemma \ref{lemma:KCartan} implies that $Q$ determines a canonical Cartan involution $\theta_Q$ on the abstract Cartan $\Lieh$ using specialization . Denote by $\Liea_Q$ the $(-1)$-eigenspace of $\theta_Q$ so that $\Lieh = \Liet_Q \oplus \Liea_Q$. Let $\phi$ be an irreducible $K$-homogeneous connection on $Q$. For any given $\lambda \in \Lieh^*$ we write 
  \[ \lc = \lambda|_{\Liet_Q} = \frac{1}{2}(\ld + \theta_Q \ld), \quad \lnc = \lambda|_{\Liea_Q} = \frac{1}{2}(\ld - \theta_Q \ld). \] 
Let $x \in Q$ and $\tau = T_x(\phi)$ be the geometric fiber of $\phi$ at x. Then $\tau$ is an irreducible finite dimensional representation of the stabilizer $K_x$ of $x$ in $K$. The connection $\phi$ is completely determined by this representation of $K_x$ on $T_x(\phi)$ since $\phi$ is $K$-homogeneous. Let $\Liec$ be a $\theta$-stable Cartan subalgebra in the Borel subalgebra $\Lieb_x$. The Lie algebra $\Liek_x = \Liek \cap \Lieb_x$ of $K_x$ is the semidirect product of the \emph{toroidal part} $\Liek \cap \Liec \cong \Liet_Q$ with the nilpotent radical $\Lieu_x = \Liek \cap \Lien_x$ of $\Liek_x$. Let $U_x$ be the unipotent subgroup of $K$ corresponding to $\Lieu_x$. It is the unipotent radical of $K_x$. Let $T$ be the Levi factor of $K_x$ with Lie algebra $\Liet_Q$, then $K_x$ is the semidirect product of $T$ with $U_x$. The representation of $K_x$ in $T_x(\phi)$ is trivial on $U_x$, so it can be viewed as a representation of $T$. Note that the assumption on $\GR$ at the beginning of Section \ref{subsec:Dmod} implies $\phi$ is always a line bundle, but with more general $\GR$ the component groups of Cartan subgroups could be nonabelian and so $\phi$ could be of higher rank.

\begin{definition}
We say that the $K$-connection $\phi$ is \emph{compatible} with $\lambda-\rho$ if the differential $d \tau$ of the $K_x$-representation $\tau$ decomposes into a direct sum of a finite number of copies of the one dimensional representation determined by the restriction of $\lambda-\rho$ (specialized to $\Liec$) to $\Liet_Q$. 
\end{definition}

For $i: Y=Q \hookrightarrow X$, an alternative description of $\Dl^i$ similar to that of $\Dl$ is as follows. Use the same notations $\LLiek$, $\LLieb$, $\LLien$, etc., for their restrictions to $Q$ as $\OO_Q$-modules. Then $\LLiek$, $\LLieb$ and $\LLien$ are still Lie algebroids on $Q$ (the achor maps for $\LLieb$ and $\LLien$ are zero). The sheaf $\LLiek \cap \LLien$ is a subsheaf of Lie ideals of $\LLiek$ over $Q$. Define the sheaf of algebras $\Dh^i = \U (\LLiek / \LLiek \cap \LLien)$ over $Q$. Then $\LLiet_Q = \LLieh \cap (\LLiek / \LLiek \cap \LLien)$ and hence the Lie algebra $\Liet_Q$ lies in the center of (the sections of) the sheaf of algebras $\Dh^i$. Let $\lambda_Q = (\lambda-\rho)|_{\Liet_Q} = \lc - (\rho|_{\Liet_Q}) \in \Liet^*_Q$. Then
      \[  \Dl^i = \Dh^i \otimes_{\U \Liet_Q} \C_{\lambda_Q}. \]
 In particular, $\Dl^i$ only depends on $\lc$. It is clear from this description that a $K$-connection $\phi$ compatible with $\lambda-\rho$ is naturally a (left) module over $\Dl^i$. Regard $\DQ$ as a right $\U\LLiek$-module where the action of $\U\LLiek$ factors through $\Dl^i$. Then we have the identification
         \[ i_+ \phi = i_*(\DQ \otimes_{\U\LLiek} \phi). \]

\begin{definition}
 The \emph{standard Harish-Chandra sheaf} associated to the triple $(\lambda, Q,\phi)$ is the $(\Dl,K)$-module
   \[  \ilq := i_+ \phi. \]
\end{definition}

Under certain conditions, $\ilq$ is irreducible as $\Dl$-module and its sections produces an irreducible $(\Lieg,K)$-module. In general it contains a unique irreducible Harish-Chandra subsheaf (Lemma 6.6, \cite{Milicic}), denoted by $\llq$. With certain assumptions on $\lambda$ and $Q$, the cohomologies of such $\llq$ give a geometric classification of the admissible representations of $G_\R$. See Section 9, \cite{Localization}. 

\subsection{Geometric classification of irreducible Harish-Chandra modules} We are state the results on the classification of irreducible tempered representations in terms of Harish-Chandra modules.

\begin{definition}
 A triple $(\lambda,Q,\phi)$ is called \emph{regular} if $\lambda$ is dominant and $\Gamma(X,\llq) \neq 0$.
\end{definition}

\begin{definition}
  The standard module $\Gamma(X,\ilq)$ associated to the data $(\lambda,Q,\phi)$ is called \emph{basic} if
  \begin{enumerate}
    \item
      $\lambda$ is dominant and $\Gamma(X,\ilq)$ is nontrivial, and
    \item
      $\lnc= \frac{1}{2}(\lambda-\theta_Q \lambda) \in \Lieh^*$ is purely imaginary. 
  \end{enumerate}
  
  A basic standard module $\Gamma(X,\ilq)$ is called \emph{final} if the data $(\lambda, Q,\phi)$ is regular.
\end{definition}

\begin{theorem}[\cite{Chang}, \cite{Localization}]\label{thm:Dtempered}
Every final basic standard module is an irreducible tempered Harish-Chandra module. Conversely, any irreducible tempered Harish-Chandra module is isomorphic to a final basic standard module $\Gamma(X,\ilq)$.
\end{theorem}

%In general a given irreducible tempered representation can be realized as global sections of standard Harish-Chandra sheaves associated to different data $(\lambda, Q,\phi)$.

Mirkovic also proved a similar result in Lemma 4.1.4, \cite{Mirkovic}. He showed that an irreducible tempered representation can be realized as global sections of standard Harish-Chandra sheaves on different $K$-orbits and they are related by intertwining functors.

\subsection{$\KC$-orbits on the flag variety}\label{subsec:Korbit}

The geometry of $\KC$-orbits on the flag variety $X$ is quite involved and plays an important role in representation theory and the study of Kazhdan-Lusztig-Vogan polynomials (\cite{LusztigVogan}, \cite{VoganKLV}). We only state properties necessary for the discussions on minimal $\KC$-types in \S \ref{subsec:minKtypes}. The results needed here are from \S 3 and \S 6 of \cite{Chang}. Some of the original literature are \cite{Matsuki1}, \cite{Matsuki2} and \cite{Springer}. It should be mentioned that the only possibly novel perspectives are contained in Definition \ref{defn:special_revised} and Proposition \ref{prop:secM}, which might be well-known to the experts but we can not find in the literature. 

Give a $K$-orbit $Q$ of $X$, choose a point $x \in Q$ and denote $\Lieb_x$ its corresponding Borel subalgebra of $\Lieg$. By Lemma \ref{lemma:KCartan}, there exists a $\theta$-stable Cartan subalgebra $\Liec$ of $\Lieb_x$. Let $R=R(\Lieg,\Liec)$ be the root system of $(\Lieg,\Liec)$ in $\Liec^*$. The root system $(\Liec, R)$ is canonically identified with the abstract root system $(\Lieh, \Sigma)$ via specialization. Then the Borel subalgebra $\Lieb_x$ together with $\Liec$ determines a positive root system $R^+_x$ in $R$. Now suppose $w$ is some element in the abstract Weyl group $W$. Using the specialization at $x$ we identify $W$ with the Weyl group $W(\Lieg,\Liec)$ associated to $(\Lieg,\Liec)$. Then $w R^+_x$ determines a new Borel subalgebra $\Lieb_y$ corresponding to a point $y \in X$. Let $w Q = K \cdot y \subset X$ be the $K$-orbit of $y$. By Lemma \ref{lemma:KCartan} we see that $w Q$ does not depend on the choice of $x \in Q$ and $\Liec$, even though $y$ does. Recall the following definition.
\begin{definition}
 For a given $\theta$-stable Cartan subaglebra $\Liec$, a root $\alpha$ in $R(\Lieg,\Liec)$ is called \emph{real} if $\theta \alpha = -\alpha$; \emph{complex} if $\theta \alpha \neq \pm \alpha$; \emph{compact} (resp. \emph{noncompact}) \emph{imaginary} if $\theta \alpha = \alpha$ and its corresponding root space lies in $\Liek$ (resp. $\Lies$).
\end{definition}

 Consider the fibration $\pi_\alpha : X \to X_\alpha$ associated to a simple root $\alpha \in R^+_x$. Each fiber of $\pi_\alpha$ is isomorphic to $\PP$. For a simple root $\alpha$ and a subset $Y$ of $X$, we denote by $P_\alpha * Y$ the set $\pi\pb_\alpha(\pi_\alpha(Y))$. In particular, for any point $x \in X$, denote by $X_x = P_\alpha * \{ x \} = \pi\pb_\alpha(\pi_\alpha(x))$ the fiber of $\pi_\alpha$ containing $x$. We denote $Q_\alpha := \pi_\alpha(Q)$, which is a single $K$-orbit in $X_\alpha$.

\begin{lemma}[Lemma 3.6, \cite{Chang}]\label{lemma:flagfiber}
 Consider the $K$-orbit $Q= K \cdot x$.
  \begin{enumerate}
    \item
      If $\alpha$ is a complex simple root and $\theta_Q \alpha \in \Sigma^+$, then $X_x \cap Q = \{ x \}$.
    \item
      If $\alpha$ is a complex simple root and $\theta_Q \alpha \notin \Sigma^+$, then $X_x \cap (K \cdot x) = X_x - \{ y \}$, where $y$ is a unique point in $X_x$.
    %\item
      %If $\alpha$ is real, then $L_x \cap (K \cdot x) = L_x - \{ y_+, y_- \}$, where $y_+$, $y_-$ are two distinct points.
  \end{enumerate}
\end{lemma}

\begin{corollary}
  If $\alpha$ is a complex simple root and $Q$ is a $K$-orbit of X, then $P_\alpha * Q$ is a disjoint union of the two $K$-orbits $Q$ and $s_\alpha Q$. 
  \begin{enumerate}
    \item
      If $\theta_Q \alpha \in \Sigma^+$, then $\dim Q = \dim (P_\alpha * Q) -1$ and $\dim(s_\alpha Q) = \dim (P_\alpha * Q)$. In this case, the restriction $\pi_\alpha|_{s_\alpha Q}: s_\alpha Q \to Q_\alpha$ is an isomorphism. 
    \item
      If $\theta_Q \alpha \notin \Sigma^+$, then $\dim Q = \dim (P_\alpha * Q)$ and $\dim(s_\alpha Q) = \dim (P_\alpha * Q) - 1$. In this case, the restriction $\pi_\alpha|_{Q}: Q \to Q_\alpha$ is an isomorphism.
   \end{enumerate}
\end{corollary}

\begin{definition}
   A positive root $\alpha$ is said to be \emph{$\theta$-stable} (resp. \emph{$(-\theta)$-stable}) if $\theta(\alpha) $ (resp. $-\theta(\alpha)$) is positive. A positive root system is said to be \emph{$\theta$-stable} (resp. \emph{$(-\theta)$-stable}) (\emph{outside real roots})  if $\alpha$ is $\theta$-stable (resp. $(-\theta)$-stable) for any complex positive root $\alpha$. A $K$-orbit $Q$ is said to be \emph{$\theta$-stable} (resp. \emph{$(-\theta)$-stable}) (\emph{outside real roots}) if $R^+_x$ is $\theta$-stable (resp. $(-\theta)$-stable) (outside real roots) for some $x \in Q$ (Surely this definition is independent of the choice of $x$).
\end{definition}

\begin{definition}\label{defn:distinguished}
  A $K$-orbit $Q_b$ is a \emph{distinguished} ($\theta$-stable) orbit associated to a $K$-orbit $Q=K \cdot x$ if $Q_b = K \cdot y$ for some $y \in X$ such that
  \begin{enumerate}
    \item
      There is a $\theta$-stable Cartan subalgebra $\Liec$ in $\Lieb_x \cap \Lieb_y$. Let $R^+_x$ and $R^+_y$ be the positive root systems of $\Liec$ determined by $\Lieb_x$ and $\Lieb_y$ respectively as above;
    \item
      $R^+_y$ is $\theta$-stable outside real roots;
    \item
      All the real roots and $\theta$-stable roots in $R^+_x$ is contained in $R^+_y$.
  \end{enumerate}
\end{definition}

 The following theorem shows that distinguished orbits always exist.

\begin{theorem}\label{thm:distinguished}
  Let $Q = K \cdot x$ be a $K$-orbit.
  \begin{enumerate}
    \item
      There exits a sequence of roots $\alpha_1, \ldots, \alpha_n$ in $R^+_x$ such that, for $i=1, \ldots, n$, $\alpha_i$ is a complex $(-\theta)$-stable simple root in $s_{\alpha_{i-1}} \cdots s_{\alpha_1} R^+_x$. Moreover, the set of non-real roots in $s_{\alpha_n} \cdots s_{\alpha_1} R^+_x$ is $\theta$-stable.
    \item
      Let $\alpha_1, \ldots, \alpha_n$ be as in (1) and denote $w_b = s_{\alpha_n} \cdots s_{\alpha_1}$. Then $Q_b=w_b Q$ is a distinguished orbit associated to $Q$. Moreover, we have
      \[ \cl{Q} = P_{\alpha_1} * \cdots * P_{\alpha_n} * \cl{Q}_b.  \]
  \end{enumerate}
\end{theorem}

\begin{proof}
  (1) is case (1) of Lemma 6.4, \cite{Chang}. Then the first part of (2) follows from Definition \ref{defn:distinguished}. The second part of (2) is Theorem 6.7, \cite{Chang}.
\end{proof}

Let $Q_b = K \cdot y$ be the distinguished orbit associated to $Q$ as in Theorem \ref{thm:distinguished}. Then $R^+_y$ is $\theta$-stable outside real roots. Let $S$ be the set of all simple real roots in $R^+_y$, we hence have a fibration $\pi_S: X \to X_S$. Denote $Q_r = \pi_S(Q_b)$. Denote by $P$ the parabolic subgroup corresponding to $z = \pi_S(y)$, by $L$ the Levi factor of $P$ and by $L_s = [L, L]$ the semisimple part of $L$. Then $P$ is $\theta$-stable and $L_s$ is split. Therefore $Q_r$ is a closed $K$-orbit in $X_S$ and $\cl{Q}_b = \pi\pb_S(Q_r)$ ((6.6) of \cite{Chang}). Together with (1) and (2) of Lemma \ref{lemma:flagfiber}, this implies that there exists a natural $K$-equvariant fibration $\pi: Q \to Q_r$.  Let $T$ be the reductive part of the isotropy group $K_x$. then $T_1 = T \cap L_s$ is finite and $K \cap L / T \cong K \cap L_s / T_1$ is an open orbit in 
  \[  \pi\pb (z) \cong L/T \cong L_s / T_1. \]
Let $U$ be the exponential of the span of $E_{\alpha_i} + \theta E_{\alpha_i}$ for $\alpha_i$ in (1) of Theorem \ref{thm:distinguished}  ($E_{\alpha_i}$ is an $\alpha_i$-root vector), then 

\begin{lemma}[(6.8)', \cite{Chang}]\label{lemma:fiber_pi}
  The fiber over $z$ of the fibration $\pi:Q \to Q_r$ is canonically isomorphic to $(K \cap L_s/T_1) \times U$.
\end{lemma}

Now consider a set of data $(\ld, Q=K \cdot x, \phi)$ with $\ld$ dominant. Recall we have the decomposition $\ld = \lc + \lnc$, where $\lc=\frac{1}{2}(\ld + \theta_Q \ld)$ and $\lnc=\frac{1}{2}(\ld - \theta_Q \ld)$.

\begin{definition}[Defn. 6.9, \cite{Chang}]
  An orbit $Q_b = K \cdot y$ is a \emph{$\lc$-distinguished} ($\theta$-stable) orbit associated to $(\ld, Q,\phi)$ if $Q_b$ is distinguished associated to $Q$ as in Definition \ref{defn:distinguished}, and $\lc$ (considered as an element in $\Lieh^*_x=\Lieh^*_y$) is $R^+_y$ dominant.
\end{definition}

\begin{proposition}[Prop. 6.10, \cite{Chang}]\label{prop:ld_disting}
  For any set of data $(\ld, Q, \phi)$ with $\ld$ dominant, there exists at least one associated $\ld_+$-distinguished orbit.
\end{proposition}

\subsection{Minimal $\KC$-types}\label{subsec:minKtypes}

We recall the definition and properties of \emph{special $\KC$-types} of standard modules from \S 8 of \cite{Chang}. It was shown there that special $\KC$-types coincide with Vogan's definition of minimal $K$-type (Definition \ref{defn:minimal}).

The notion of special $K$-types is a generalization of that of \emph{fine $\KC$-types} introduced by Bernstein-Gelfand-Gelfand in \cite{BGG1} and \cite{BGG2}, which we now recall. Let $\GR$ be a split group and $(\ld, Q,\phi)$ be a set of data with $Q=K \cdot x$ dense in $X$. Then $K_x$ is finite and $\Gamma(\ilq)=\Gamma(Q,\phi)$. Choose a Cartan subgroup $H_x \cong B_x / N_x$ with Lie algebra $\Lieh_x$ and by $R^+_x$ the associated positive root system on $\Lieh^*_x$. In this case all roots are real. For any root $\alpha \in R^+_x$, choose $\alpha$-root vector $E_\alpha$ and $(-\alpha)$-root vector $E_{-\alpha}$ in $\Lieg_\R$ such that 
  \[ \theta E_\alpha = E_{-\alpha}, [ E_\alpha, E_{-\alpha} ] = -H_\alpha, [H_\alpha, E_\alpha] = 2 E_\alpha, \quad \text{and}~ [H_\alpha, E_{-\alpha}] = -2 E_{-\alpha}. \]
Let 
  \[ Z_\alpha = -i(E_\alpha+E_{-\alpha}), \]
then $Z_\alpha \in \Liek$. Moreover, $\{ Z_\alpha | \alpha \in R^+_x \}$ forms a basis of $\Liek$.

\begin{definition}\label{defn:fineK1}
  A $K$-type $\sigma$ is said to be \emph{fine} if the eigenvalue of $\sigma(Z_\alpha)$ lies in the range $[-1,1]$ for each simple root $\alpha$.
\end{definition}

Geometrically, consider the fibration $\pi_\alpha : X \to X_\alpha$ associated to a simple root $\alpha$ and recall that $X_y = \pi_\alpha \pb (\pi_\alpha(y))$ for each $y \in X$. We use $z$ as coordinate on $X_y$ ($y \in Q$) such that $X_y \cap Q = \{ z|z \neq 0, \infty  \}$ (Lemma \ref{lemma:flagfiber}, (3)). Then Definition \ref{defn:fineK1} can be reformulated as

\begin{definition}[Defn. 8.1', \cite{Chang}]\label{defn:fineK2}
  Suppose $G$ is a split group and $Q$ is the open $K$-orbit of $X$. A section $s \in \Gamma(\ilq)=\Gamma(Q,\phi)$ is called \emph{fine} on $Q$ if for each $y \in Q$ and simple (real) root $\alpha$, the restriction of $s$ on $X_y$ is contained in the linear span of $1$, $z^{1/2}$ and $z^{-1/2}$.
\end{definition}

\begin{proposition}[Prop. 8.2., \cite{Chang}]
  Suppose $G$ is a split group and $Q$ is the open $K$-orbit of $X$. A fine $K$-type $\sigma$ occurs as a subspace $V_\sigma$ in $\Gamma(Q,\ilq)$ if and only if every section $f \in V_\sigma$ is fine.
\end{proposition}

To prepare for the general definition of special $\KC$-types, we need two ingredients. First of all, for a given standard module $\I=\ilq$, there is a natural filtration defined by the degree of derivatives transversal to $Q$,
  \[ \I^0 \subset \I^{1} \subset \cdots \subset \bigcup_{k=0}^\infty \I^k = \I,  \]
where $\I^k$ is the subsheaf of local sections of $\I$ with degree of transversal derivatives less or equal than $k$ ($k \in \ZZ$, $k \geq 0$). In particular, 
  \[ \I^0 = \phi \otimes_{\OO_Q} \omega_{Q/X}.  \]

The second ingredient is related to the distinguished orbit in Definition \ref{defn:distinguished}. Let $Q_b$ be a distinguished $\theta$-stable orbit associated to $Q$. Recall from \S \ref{subsec:Korbit} there is a $K$-equivariant fibration $\pi: Q \to Q_r$ whose fiber over $z \in Q_r$ is canonically identified with $(K \cap L_s)/T_1 \times U$ (Lemma \ref{lemma:fiber_pi}) and $(K \cap L_s)/T_1$ is an open orbit in $\pi\pb_S(z)$. By Theorem \ref{thm:distinguished}, we have
  \begin{equation}
    \omega_{Q/X} = \pi^* \omega_{Q_r/X_S}.
  \end{equation}
Thus we have
  \begin{equation}
     \pi_* (\phi \otimes \omega_{Q/X})  |_z \cong \Ind_{T_1}^{K \cap L_s} \left( \tau \otimes [\omega_{Q_r/X_S}]|_{T_1}) \right) \otimes_\C \C[U],
  \end{equation}
where $[\omega_{Q_r/X_S}]$ denotes the $K \cap L$-module given by the action of $K \cap L \subset K_z$ on $\omega_{Q_r/X_S}|_z$. The same convention will be adopted in similar contexts.

\begin{definition}[Defn. 8.4, \cite{Chang}]\label{defn:special_qb}
  A section $s \in \Gamma(X,\iqzero)=\Gamma(Q, \phi \otimes \omega_{Q/X})$ is called \emph{$Q_b$-special} if for each $z \in Q_r$ the restriction of $s$ on $\pi\pb(z) = (K \cap L_s/T_1) \times U$ is constant along $U$ and is fine on $K \cap L_s / T_1$. A $K$-type $\sigma$ occurring as a subspace $V_\sigma$ in $\Gamma(X, \iqzero)$ is called \emph{$Q_b$-special} if each section $s \in V_\sigma$ is $Q_b$-special.
\end{definition}

\begin{definition}[Defn. 8.5, \cite{Chang}]
  A $K$-type $\sigma$ which occurs as a subspace $V_\sigma$ in $\Gamma(X, \ilq)$ is called \emph{special} if $V_\sigma \subset \Gamma(X,\iqzero)$ and $V_\sigma$ is $Q_b$-special for every $\lc$-distinguished orbit associated to the data $(\ld,Q,\phi)$.
\end{definition}

We restate this definition in an equivalent way which is convenient for later application.

\begin{definition}\label{defn:special_revised}
 A $K$-type $\sigma$ which occurs as a subspace $V_\sigma$ in $\Gamma(X, \ilq)$ is called \emph{special} if $V_\sigma \subset \Gamma(X,\iqzero)$ and any section $s \in V_\sigma$ restricted to each fiber $X_y$ of the fibration $\pi_\alpha: X \to X_\alpha$ for any simple root $\alpha$ satisfies the following:
 \begin{enumerate}
   \item
     If $\alpha$ is a real root, then $s|_{X_y}$ is contained in the linear span of $1$, $z^{1/2}$ and $z^{-1/2}$ (cf. Defn. \ref{defn:fineK2});
   \item
     If $\alpha$ is a complex root and $-\theta \alpha \in R^+_x$, then $s|_{X_y}$ is constant along $X_y$ (cf. Defn. \ref{defn:special_qb}). 
 \end{enumerate}
\end{definition}

\begin{proposition}[Prop. 8.6, \cite{Chang}]\label{prop:qb_disting}
  Suppose $Q_b$ is a $\lc$-distinguished orbit associated to the data $(\ld, Q=K\cdot x ,\phi)$ and assume that $\langle \ld, \alpha \rangle \neq 0$ for all compact simple roots $\alpha$, then $Q_b$-special $K$-types exist in $\Gamma(X,\ilq)$.
\end{proposition}

\begin{proposition}[Prop. 8.12, Corollary 8.13, \cite{Chang}]\label{prop:special_indep}
  If $Q_b$ and $Q'_b$ are two $\lc$-distinguished orbits associated to the data $(\ld,Q,\phi)$, then the set of $Q_b$-special $K$-types coincides with that of $Q'_b$-special $K$-types. In particular, a $K$-type $V_\sigma \subset \Gamma(X,\ilq)$ is special if and only if it is $Q_b$-special for one $\lc$-distinguished orbit $Q_b$ associated to $(\ld, Q, \phi)$.
\end{proposition}

By Proposition \ref{prop:ld_disting}, \ref{prop:qb_disting} and \ref{prop:special_indep}, we have
\begin{corollary}[Cor. 8.14, \cite{Chang}]
  If $\ld$ is dominant and $\Gamma(X, \ilq)$ is not trivial, then special $K$-types exist in $\Gamma(X,\ilq)$.
\end{corollary}

\begin{theorem}
  If $(\ld, Q,\phi)$ is regular, then all the special $K$-types occur in $\Gamma(X,\llq)$.
\end{theorem}

Eventually, we have the following characterization of minimal $K$-types of the standard module $\Gamma(X,\ilq)$.

\begin{theorem}[Theorem 8.15, 8.18, 8.19, \cite{Chang}]\label{thm:minimal-K}
 Let $\Gamma(X, \ilq)$ be a nontrivial standard module with $\lambda$ dominant. Then the special $K$-types of $\Gamma(X, \ilq)$ coincide with its special/minimal $K$-types. Moreover, they all occur in $\Gamma(Q, \omega_{Q/X} \otimes_{\OO_Q} \phi)$ and $\Gamma(X,\llq)$. Each of them occurs with multiplicity $1$.
\end{theorem} 

We now state the main result of this section. 

\begin{definition}
Define $\MQ$ to be the coherent $K$-equivariant $\OO_X$-submodule of $i_* \phi$ generated by sections which are contained in minimal $K$-types of $\Gamma(X,\ilq)$.
\end{definition}

\begin{proposition}\label{prop:secM}
  Under the assumption of Theorem \ref{thm:minimal-K}, the space of global sections $\Gamma(X,\MQ)$ as a $K$-representation is the direct sum of minimal $K$-types of $\Gamma(X,\ilq)$.
\end{proposition}

\begin{proof}
  Suppose $V_\sigma$ is a $K$-type contained in $\Gamma(X,\MQ) \subset \Gamma(X,\ilq)$. Then the restriction of any section $s \in V_\sigma$ to each fiber $X_y \cong \PP$ of the fibration $\pi_\alpha: X \to X_\alpha$ for any simple root $\alpha$ is a section of the restriction sheaf of $\MQ$ to $X_y$, which is in turn the $\OO_{X_y}$-module generated by restriction of (sections of) minimal or special $K$-types of $\Gamma(X,\ilq)$ to $X_y$. We only need to show that $s$ satisfies the conditions in Definition \ref{defn:special_revised}. Indeed, from that definition we deduce that:
  \begin{enumerate}
    \item
      If $\alpha$ is complex and $-\theta \alpha \in R^+_x$, $\MQ|_{X_y}$ is just $\OO_{X_y}$ and hence $s |_{X_y}$ must be constant;
    \item
      If $\alpha$ is real, then $\MQ|_{X_y}$ is either isomorphic to $\OO_{X_y}$ or $\OO_{X_y}(1)$. In either case, $s$ must be contained in the linear span of $1$, $z^{1/2}$ and $z^{-1/2}$.
  \end{enumerate}
  Hence again by Definition \ref{defn:special_revised}, $s$ is special (or $V_\sigma$ is special).

\end{proof}

\section{Deformation of Harish-Chandra modules}\label{sec:deform}

\subsection{Contraction of Harish-Chandra pair}

\begin{definition}
Given a Harish-Chandra pair $(\Lieg,K)$, its associated \emph{Cartan motion pair} is the Harish-Chandra pair $(\Lieg_0,K)$ where $\Lieg_0$ is the Lie algebra
  \[  \Lieg_0 := \Liek \ltimes \Lies, \]
where 
  $ \Lies := \Lieg / \Liek $
is considered as a $K$-representation.
\end{definition}

The Lie algebras $\Lieg_0$ and $\Lieg$ fit into an algebraic family of Lie algebras $\Lieg_t$, $t \in \C$, with the fiber at $t=0$ being the Lie algebra $\Lieg_0$ and other fibers $\Lieg_t$, $t \neq 0$, isomorphic to $\Lieg$. This is the well-known `deformation to the normal cone' construction. More precisely, we take the trivial vector bundle $\C \times \Lieg$ and regard it as a sheaf of $\OO_\C$-modules over the affine line $\C$. It is a sheaf of Lie algebras and its module of global sections is $\Lieg[t]=\Lieg \otimes_\C \C[t]$, where $t$ is the coordinate function of $\C$. Let $\Lieg_t$ be the subsheaf of Lie subalgebras in $\C \times \Lieg$ whose germs of sections consist of those taking values in $\Liek \subset \Lieg$ at $0 \in \C$. This definition of $\Lieg_t$ was communicated to the author by Nigel Higson. We take $\KK=K(\C[t])=K \times \C$ to the constant family of groups over $\C$ with fiber $K$.

\begin{definition}
The \emph{contraction of the Harish-Chandra pair $(\Lieg,K)$} is the pair $\gkt$.
\end{definition}

This is a special case of \emph{algebraic family of Harish-Chandra pairs} in the sense of \cite{BHS1} where the base space is $\C$. In that paper general families of groups are allowed, whereas we only consider the constant family $\KK$ with fiber $K$. Also in our case any quasi-coherent sheaf over the base space $\C$ is equivalently a $\C[t]$-module. In the setting of the paper, there is a Cartan involution $\theta$ on $\Lieg$ and we can identify $\Lies$ with the $(-\theta)$-fixed part of $\Lieg$ so that $\Lieg=\Liek \oplus \Lies$. Then we have $\Lieg_t = \Liek[t] \oplus t \Lies[t] \subset \Lieg[t]$. 

The notion of \emph{admissible (algebraic) family of Harish-Chandra modules} over an algebraic family of Harish-Chandra pairs are also defined in 2.4.1, \cite{BHS1}. We adopt their terminologies for the pair $\gkt$ yet with a slight modification. Recall that an \emph{action of $\KK$ on a $\C[t]$-module $\F$} is a morphism of $\C[t]$-modules
  \[  \F \to \OO_{\KK} \otimes_{\C[t]} \F \]
that is compatible with the multiplication and inverse operations on $\KK$ in the usual way. 

\begin{definition}\label{defn:KKaction}
  A \emph{family of representations} of $\KK$, or simply a \emph{representation} of $\KK$, is a $\C[t]$-module $\F$ that is equipped with an action of $\KK$. We say that $\F$ is flat if it is flat as a $\C[t]$-module and that $\F$ is \emph{admissible} if it is flat and
   \[ \lhom_{\KK} (\C[t] \otimes_\C V, \F) \cong [V^* \otimes_\C \F ]^\KK \]
is a free $\C[t]$-module of finite rank for every finite-dimensional representation $V$ of $K$. In this case there is a canonical isotypical decomposition
  \[  \F \cong \bigoplus_{\sigma \in \widehat{K}} \F_\sigma, \]
indexed by the equivalence classes of irreducible algebraic representations of $K$, where
  \[  \F_\sigma = V_{\sigma} \otimes_\C \lhom_{\KK} (\C[t] \otimes_\C V_\sigma, \F). \]
\end{definition}

\begin{definition}
  An \emph{algebraic family of Harish-Chandra modules} for $\gkt$, or simply a \emph{$\gkt$-module}, consists of a $\C[t]$-module $\F$, that is equipped with an action $\KK$ on $\F$ and an action of $\Lieg_t$ on $\F$, such that the action morphism
   \[ \Lieg_t \otimes_{\C[t]} \F \to \F  \]
is $\KK$-equivariant, and such that the differential of the $\KK$-action is equal to the restriction of the $\Lieg_t$-action to $\Liek[t]$. A $\gkt$-module $\F$ is said to be flat if it is flat as a $\C[t]$-module. 
\end{definition}

\begin{definition}\label{defn:gktmod_adm}
  A flat $\gkt$-module $\F$ is called \emph{admissible} if $\F$ is finitely generated as a $\Lieg_t$-module, and if the $\KK$-action on $\F$ is admissible in the sense of Definition \ref{defn:KKaction}. 
%The category of all admissible $\gkt$-modules is denoted by $\Mod_{ad}(\Ug_t,\KK)$, while the category of all (not necessarily admissilbe) $\gkt$-modules is denoted by $\Mod(\Ug_t,\KK)$.
\end{definition}

\begin{remark}
  In \cite{BHS1}, all $\gkt$-modules are defined to be flat as $\C[t]$-modules. We do not assume $\gkt$ to be flat in general, however, since flat modules do not form an abelian category.
\end{remark}

\begin{remark}
  Since our base space is always $\C$, flatness is equivalent to torsion-freeness. Hence any $\C[t]$-submodule of a flat $\gkt$-module $\F$ which is invariant under the $\Lieg_t$-action and the $\KK$-action is also a flat $\gkt$-module. It is natural to call them \emph{algebraic subfamilies of Harish-Chandra modules}, or simply \emph{$\gkt$-submodules}, of the $\gkt$-module $\F$. Similarly $\gkt$-submodules of admissible $\gkt$-modules are always admissible. However the category of admissible $\gkt$-modules is not an abelian category.
\end{remark}

We will see later that MHA bijection is essentially about the study of $(\Lieg_t,\KK)$-modules. However, an alternative description of $(\Lieg_t,\KK)$ is more convenient for the approach adopted in this paper. Consider the Lie subalgebra 
  \[ \Lier := t \Lieg[t]  \subset \Lieg_t \subset \Lieg[t] \]
over $\C[t]$. Then $\Lier$ contains the Lie subalgebra $\Liek_t : = t\Liek[t] \subset \Liek[t]$ and carries the adjoint action of $\KK$.

\begin{definition}
  A \emph{$(\Lier,\KK)$-module} consists of a $\C[t]$-module $\F$ that is equipped with an action $\KK$ on $\F$ and an action of $\Lier$ on $\F$, such that the action morphism 
  \[ \Lier \otimes_{\C[t]} \F \to \F  \]
is $\KK$-equivariant, and such that the restriction of the differential of the $\KK$-action to $\Liek_t$ is equal to the restriction of the $\Lier$-action to $\Liek_t$.
\end{definition}

A $(\Lieg_t,\KK)$-module is naturally a $(\Lier,\KK)$-module by restriction of the $\Lieg_t$-action to $\Lier$. Conversely, given a $\Lier$-module $\F$, the differential of its $\KK$-action together with the $\Lier$-action give a $\Lieg_t$-action since $\Lieg_t = \Liek[t] + \Lier$, and hence $\F$ is a $\gkt$-module. Therefore the notion of a $\rkt$-module is equivalent to that of a $\gkt$-module.

\begin{definition}\label{defn:rktmod_adm}
  A $\rkt$-module is called \emph{flat} if it is flat as a $\C[t]$-module. A flat $\rkt$-module $\F$ is called \emph{admissible} if it is admissible as a $\gkt$-module.
\end{definition}

\begin{lemma}
  A $\rkt$-module $\F$ is admissible if and only if it is finitely generated as a $\rkt$-module, and if the $\KK$-action on $\F$ is admissible.
\end{lemma}

\begin{proof}
  We only need to show that if $\F$ is finitely generated as $\Lieg_t$-module and is admissible as $\KK$-representation, then $\F$ is also finitely generated as $\Lier$-module. Assume $\F = \U\Lieg_t \cdot U$ where $U$ is a finite-dimensional complex vector subspace of $\F$. Take $U'$ to be the $\KK$-subrepresentation of $\F$ generated by $U$, then $\F = \U\Lier \cdot U'$. Since $\F$ is admissible as $\KK$-representation, $U'$ is a $\C[t]$-module of finite rank. Therefore $\F$ is finitely generated as $\Lier$-module.
\end{proof}

\subsection{Algebraic family of Harish-Chandra sheaves}\label{subsec:deform}

We extend the Beilinson-Bernstein localization of $\gk$-modules to the case of $\gkt$-modules in a similar manner as \S~\ref{sec:Dmod}. Form the trivial vector bundle $\LLieg[t] = \LLieg \otimes_\C \C[t]$ and and similarly its subsheaf $\LLieg_t$. They are both Lie algebroids over the sheaf $\OO_X[t]$ of commutative algebras over $X$ in the same way as $\LLieg$ being Lie algebroid over $\OO_X$.

We now assume $\Liek$ is the fixed Lie subalgebra of a Cartan involution $\theta$ of $\Lieg$. As in \S~\ref{subsec:HC}, a given $K$-orbit $Q \xrightarrow{i} X$ determines a Cartan involution $\theta_Q$ of $\Lieh$ so that the fixed subspace is $\Liet_Q$. Denote by $\Liea_Q$ the $(-1)$-eigenspace of $\theta_Q$ so that $\Lieh = \Liet_Q \oplus \Liea_Q$. For any given $\lambda \in \Lieh^*$ we write $\lc = \lambda|_{\Liet_Q}$ and $\lnc = \lambda|_{\Liea_Q}$. Define a character 
  \begin{equation}\label{eq:lambda_t}
     \lt : = \lc + (\lnc / t) - \rho 
   \end{equation}
of $t \Lieh[t]$ (as an element in $\Liet_Q^* \oplus t\pb  \Liea^*_Q$). 

Consider another Lie subalgebra $\Lier = t \Lieg[t]  \subset \Lieg_t \subset \Lieg[t]$ over $\C[t]$ and the corresponding Lie subalgebroid $\LLier=t\LLieg[t] \subset \LLieg_t \subset \LLieg[t]$. 
We set
  \[ \Dh(\Lier) := \U(\LLier / t\LLien[t]) = \U(t (\LLieg/\LLien)[t] ),  \]
then
  \[ \U(t \LLieh[t]) = \U(t (\LLieb /\LLien) [t]) \subset \Dh(\Lier). \]
  
\begin{definition}\label{defn:drt} 
Define the sheaf of associative algebras over $X$,
  \[  \Dr := \Dh(\Lier) \otimes_{\U(t \LLieh[t])} (\C[t])_{\lt}. \]
\end{definition}  
  
\begin{definition}\label{defn:DrK}
  A (quasi-coherent) $(\Dr,\KK)$-module $\W$ is a (quasi-coherent) $\Dr$-module equipped with an equivariant $\KK$-structure, such that the action of $\LLiek_t \subset \LLier$ obtained via the left $\Dr$-module structure satisfies
  \[  (tX) \cdot s = t \frac{d}{du}\bigg|_{u=0} \exp(uX) \cdot s , \]
where $X \in \Liek$, $s$ is a local section of $\W$ and $\exp(uX)$ acts on $s$ via the $\KK$-equivariant structure. In other words, the $\LLiek_t$-module structure is compatible with the $\KK$-equivariant structure. $\W$ is coherent if it is coherent as $\Dr$-module. The category of $(\Dr,\KK)$-modules form an abelian category denoted by $\Mod(\Dr,\KK)$. 
%The subcategory of coherent $(\Dr,\KK)$-module is denoted by $\Mod_{coh}(\Dr,\KK)$.
\end{definition}  

\begin{definition}\label{defn:Drt_adm}
  A $(\Dr,\KK)$-module $\W$ is said to be flat if it is flat with respect to the projection $X \times \C \to \C$. It is said to be \emph{admissible} if it is flat and is  coherent as a $\Dr$-module. 
%The category of admissible $(\Dr,\KK)$-modules is denoted by $\Mod_{ad}(\Dr,\KK)$. (Note that it is not an abelian category.) 
\end{definition}

We now define the analogue of standard Harish-Chandra sheaves for $(\Dr,\KK)$-modules.
  
\begin{definition}
  The \emph{transfer bimodule associated to the $K$-orbit $Q$ and the character $\lt$} is the sheaf 
  \[  \DQt := i\pb(\Dr) \otimes_{i\pb \OO_{X}} \omega_{Q/X}. \]
\end{definition}

The sheaf $\DQt$ is naturally a left $i\pb \Dr$-module. Consider the Lie subalgebroid $ \LLiek_t := t \LLiek[t]$ of $\LLiek[t]$ on $X$. Then $\LLiek_t \subset \LLier$. Hence $\DrQ$ is a right $\U (\LLiek_t)$-module if we regard $\LLiek_t$ as a Lie algebroid over $Q$ by restriction. The adjoint action of $K$ on $\Lieg$ induces a $\KK$-equivariant structure on $\DQt$, which is denoted by $\Ad$. 

\begin{comment}
Consider the restriction $\tau|_Q : i^*(\LLieg/\LLien) \to i^* TX$ of the anchor map $\tau$ to $Q$. Denote $\JQ: = (\tau|_Q)\pb (TQ)$, then
  \[  \JQ = \LLieh + (\LLiek / \LLiek \cap \LLien) = \LLiea_Q + (\LLiek / \LLiek \cap \LLien). \]
\end{comment}

Now let $\phi$ be a homogeneous $K$-connection on $Q$ compatible with $\lambda-\rho$ as in \S~\ref{subsec:HC}.

\begin{definition} Define the sheaf of $\Dr$-algebras
  \[  i_\sharp \phi = \tiqt :=  i_* \left( \DQt \otimes_{\U \LLiek_t} \phi[t] \right) . \]
\end{definition}

The sheaf $i_\sharp \phi$ is a left $\Dr$-module. Moreover, it carries a $K$-equivariant structure defined by
  \[ k \cdot (P \otimes v) = (\Ad(k) P) \otimes v + P \otimes k \cdot v,  \quad \forall ~k \in K, ~ P \in \DQt, ~ v \in \phi.  \]
The $\KK$-action differentiates to a left $\LLiek[t]$-action on $i_\sharp \phi$, such that its restriction to $\LLiek_t$ coincides with the one obtained via the left $\U\LLier$-module structure of $i_\sharp \phi$ and the inclusion $\LLiek_t \subset \LLier$. Hence $i_\sharp \phi$ is a $(\Dr,\KK)$-module. Moreover, it is flat in the sense of Definition \ref{defn:Drt_adm}. 

\begin{remark}
  In \cite{YuMackeySL}, we gave a different yet equivalent construction of the family. There we localize $\LLieg_t$ to a Lie algebroid over $X$ and use the associated sheaf of universal enveloping algebras instead of $\Dr$. The advantage of the new approach is that the construction is uniform for all $K$-orbits, which makes proofs easier.
\end{remark}

Given $\phi$, there are in general different $\lambda$ which is compatible with $\phi$. To specify $\lambda$, we write $i^\lambda_+ \phi$ instead of $i_+\phi$. The following lemma follows immediately from definitions.
\begin{lemma}
  For any $s \neq 0$, denote by $\lambda_s$ the evalution of $\lt$ at $t=s$. Then 
  \begin{enumerate}
  \item
    there are canonical isomorphisms of sheaves of algebras
      \[ \Dr \big|_{t=s} \cong \D_{\lambda_s},  \quad \forall~ s \neq 0; \]
  \item 
    there are canonical $K$-equivariant isomorphisms of sheaves of $i\pb \D^{\lambda_s}$-$\U\LLiek$-bimodules
      \[ \DQt \big|_{t=s} \cong \D^{\lambda_s}_{X \gets Q}, \quad \forall~ s \neq 0;   \]
  \item
    there are canonical isomorphisms of $(\D_{\lambda_s},K)$-modules 
      \[ i_\sharp \phi \big|_{t=s} \cong i^\lambda_+ \phi, \quad \forall~ s \neq 0.  \]
\end{enumerate}
\end{lemma}

If we regard $i_\sharp \phi$ as a left $\U\LLier$-module via the natural morphism $\U\LLier \to \Dr$, the left $\U\LLier$-module structure can be extended to a left $\U\LLieg_t$-module structure via the $K$-equivariant structure by $\Lier + \Liek = \Lieg_t \subset \Lieg[t]$. Therefore the sheaf $i_\sharp \phi$ is naturally a $(\LLieg_t, \KK)$-module. However, it is not a coherent $\U\LLieg_t$-module or $\Dr$-module in general. The coherence fails when $\lnc$ is not generic, as illustrated in the case of $\SLR$ in a previous paper \cite{YuMackeySL} of the author. To fix this, recall by Theorem \ref{thm:minimal-K} that the minimal $K$-types of the stand module $\Gamma(X,\ilq)$ are included in $\Gamma(Q,\phi)$ and hence is independent of $\lnc$. Recall $\MQ$ is the coherent $K$-equivariant $\OO_X$-submodule of $i_* \phi$ generated by those minimal $K$-types.

\begin{definition}
  Define $i_\flat \phi$ to be the subsheaf of coherent $\Dr$-submodules of $i_\sharp \phi$ generated by $\MQ$. Equivalently,
  \[  i_\flat \phi = \iqt := (i_* \DQt) \otimes_{\U \LLiek_t} \MQ[t]. \]
\end{definition}

It is clear that $i_\flat \phi$ is an admissible $\Dr$-module in the sense of Definition \ref{defn:Drt_adm} since $i_\sharp \phi$ is flat and $\C[t]$ is a PID. The specialization of $i_\flat \phi$ to $t = 1$ is a coherent $(\D_\lambda,K)$-module over $X$.  Therefore its global sections form a $(\Lieg, K)$-module with infinitesimal character $\chi_\lambda$.

We define a family version of the algebra $\Ugh$. Let $\Urh$ be the subalgebra of $\Ugh[t]$ generated by $\Lier \subset \Ug[t]$ and $t\Lieh \subset \Lieh[t]$. To analyze $\Urh$, note that there is an isomorphism of commutative algebras
  \[  \gamma_t: (\U\Lier)^{\Lieg} \xrightarrow{\sim} S(t\Lieh[t])^{W,.}, \] 
which is the restriction of 
  \[  \gamma_{HC} \otimes 1:  \U(\Lieg[t])^{\Lieg} = (\Ug)^{\Lieg} \otimes_\C \C[t] \xrightarrow{\sim} (S\Lieh)^{W, .} \otimes_\C \C[t] = S(\Lieh[t])^{W, .}. \]
Here $( - )^{\Lieg}$ stands for invariant part under the adjoint $\Lieg$-action and the twisted action of $W$ on $\Lieh$ extends $\C[t]$-linearly to $\Lieh[t]$ and $t\Lieh[t]$. When $0 \neq s \in \C$, $\gamma_t|_{t=s}$ is isomorphic to the Harish-Chandra isomorphism $\gamma_{HC}$, while $\gamma_t|_{t=0}$ is isomorphic to the Chevalley isomorphism
  \[ \gamma_{Cl}: (S\Lieg)^\Lieg \xrightarrow{\sim} (S\Lieh)^W, \]
where the $W$ action is the usual one. Note that the construction of $\gamma_t$ is nothing but the Rees algebra version of the Harish-Chandra isomorphism with respect to the natural filtrations on $(\Ug)^\Lieg$ and $(S\Lieh)^{W,.}$. Therefore we can regard $S(t\Lieh[t])$ as a $(\U\Lier)^{\Lieg}$-algebra and there is an isomorphism of algebras
  \[ \Urh \cong \U\Lier \otimes_{(\U\Lier)^\Lieg} S(t\Lieh[t]).  \]

Set $I_{\lt}$ to be the kernel of the homomorphism $S(t\Lieh[t]) \to \C[t]$ determined by $\lt: t\Lieh \to \C[t]$. Then $I_{\lt}$ lies in the center of $\Urh$ and hence $I_{\lt} \Urh$ is a two-sided ideal of $\Urh$.

\begin{definition}
  $\Urlt : = \Urh / I_{\lt} \Urh = \Urh \otimes_{S(t\Lieh[t])} (\C[t])_{\lt}$, where $(\C[t])_{\lt}$ is $\C[t]$ regarded as an $S(t\Lieh[t])$-module via $\lt$.
\end{definition}

We have the analogue of Proposition \ref{prop:Dh} and Proposition \ref{prop:Uchi}.

\begin{proposition}\label{prop:Dhr}
  The natural morphisms
    \[ \Urh \to \Gamma(X, \Dh(\Lier))  \]
  and
    \[ \Urlt \to \Gamma(X, \Dr) \]
  are isomorphisms of algebras. Moreover, we have $H^i(X,\Dh(\Lier)) = H^i(X,\Dr) = 0$ for $i > 0$.
\end{proposition}

The natural homomorphism $\U\Lieg_t \to \U_{\lt}(\Lieg)$ is surjective, so $\U_{\lt}(\Lieg)$ can be regarded as a quotient algebra of $\Ug_t$. We denote $\Mod(\U_{\lt}(\Lieg),\KK)$ as the category of $\gkt$-modules whose $\U\Lieg_t$-module structure factors through $\U_{\lt}(\Lieg)$.
%denote $\Mod_{ad}(\U_{\lt}(\Lieg),\KK)$ as the subcategory of those that are admissible as $\gkt$-modules.

\begin{lemma}\
  Taking global sections defines a functor 
    \[ \Gamma(X, -): \Mod(\Dr,\KK) \to \Mod(\U_{\lt}(\Lieg),\KK). \] 
  Moreover, it sends flat $(\Dr,\KK)$-modules to flat $\gkt$-modules.
\end{lemma}

\begin{proof}
  The first statement follows from Proposition \ref{prop:Dhr}. The second statement follows from the fact that $\C[t]$ is PID and that a module over a PID is flat if and only if it is torsion-free.
\end{proof}

\begin{corollary}
  The flat $\gkt$-module $\Gamma(X,i_\flat \phi)$ lies in $\Mod(\U_{\lt}(\Lieg),\KK)$.
\end{corollary}

\begin{remark}
  We believe that the global section functor $\Gamma(X, -)$ sends admissible $(\Dr,\KK)$-modules to admissible $\gkt$-modules, and hence $\Gamma(X, i_\flat \phi)$ is an admissible $\gkt$-module. However we do not have a simple argument at the moment. 
\end{remark}

\subsection{Coherent sheaves on Grothendieck-Springer resolutions}

\begin{comment}
The global sections of the specialization $\I_0(\lambda,Q,\phi) := \I(\lt,Q,\phi) |_{t = 0}$ give a $(\Lieg_0,K)$-module, denoted by $\widetilde{M}_0(\lambda,Q,\phi)$. It can be regarded as a $K$-equivariant coherent sheaf $\EE=\EE(\lambda,Q,\phi)$ over $\Lies^*$ since the symmetric algebra $S(\Lies)$ is a subalgebra of $\U\Lieg_0$. Denote by $\supp(\EE)$ the support of $\EE$ in $\Lies^*$.

% For generic $\lambda$, the support $\supp(\EE)$ of $\EE$ in $\Lies^*$ is a single $K$-orbit, over which $\EE$ is a vector bundle, and $\widetilde{M}_0(\lambda,Q,\phi)$ is irreducible as a $(\Lieg_0,K)$-module. In general
 
Let $\Liea$ be a maximal abelian subspace of $\Lies$. We can identify $\Lies$ with $\Lies^*$ and $\Liea$ with $\Liea^*$ via the Killing form. The Chevelley Restriction Theorem says that.................  
\end{comment}

The \emph{Grothendieck-Springer simultaneous resolution} of $\Lieg^*$ (cf. \cite{ChrissGinzburg}) is the morphism,
  \[ \tilde{\mu}:  \tg \to \Lieg^*,  \]
where 
  \[  \tg = \{ (v, x) \in \Lieg^* \times X | v \in  \Lien^\bot_x, \Lien_x=[\Lieb_x,\Lieb_x] \}  \]
and the moment map $\tilde{\mu}$ is the projection to the first factor. Here $\Lien^\bot$ denotes the annihilator of $\Lien$ in $\Lieg^*$. The projection to the second factor is denoted by $\tilde{\pi}: \tg \to X$. If we fix a point $x \in X$, which corresponds to a Borel subgroup $B_x$ with Lie algebra $\Lieb_x$ and its nilpotent ideal $\Lien$, then there is an $G$-equivariant isomorphism
  \[  \tg \simeq G \times_{B_x} \Lien^{\bot}_x. \]
For any Borel subalgebra $\Lieb$ we have the short exact sequence
  \[ 0 \to \Lien \to \Lieb \to \Lieb / \Lien \to 0,  \]
and we can identify $\Lieb /\Lien$ with the abstract Cartan subalgebra $\Lieh$. Thus we have another short exact sequence
  \[  0 \to \Lieb^\bot \to \Lien^\bot \to (\Lieb/\Lien)^* = \Lieh^* \to 0. \]
The natural projections $\Lien^\bot \to \Lieh^*$ give a natural map
   \[  \nu: \tg \to \Lieh^*. \]
We regard $\lnc$ as an element of $\Lieh^*$ via the decomposition $\Lieh = \Liet_Q \oplus \Liea_Q$ determined by $\theta_Q$. Denote $\TX = \nu\pb(\lnc) \subset \tg$. It is a twisted cotangent bundle over $X$ via the  restriction $\pi: \TX \to X$ of $\tilde{\pi} : \tg \to X$ (\cite{ChrissGinzburg}). The fibers of $\tilde{\pi}$ are of the form $\Lieb^\bot_{\lnc}$, which is the preimage of $\lnc$ via the projection $\Lien^\bot \to \Lieh^*$. Then $\Lieb^\bot_{\lnc}$ is an affine subspace of $\Lien^\bot$, which is a translation of the linear subspace $\Lieb^\bot$. Via the Killing form we have identifications $\Lieg \cong \Lieg^*$, $\Lieb^\bot \simeq \Lien$, $\Lien^\bot \simeq \Lieb$, $\Lieh^* \simeq \Lieh$. Hence by regarding $\lnc$ as an element in $\Lieh$, we have the identification
\[ \tg =  \{ (v, x) \in \Lieg \times X | v \in \Lieb_x \} \]
and $\mu: \tg \to \Lieh$ is induced by the projections $\Lieb \to \Lieh$.
 
Set 
  \[  \TXK = \TX \cap \tilde{\mu}\pb (\Lies^*) \subset \tg, \] 
which is a $K$-invariant closed subvariety of $\TX$. Let
  \[  \TXQ = \pi\pb(Q) \cap \TXK \]
and let $\ti: \TXQ \hookrightarrow \TX$ be the embedding. When $\lnc = 0$, $T^*_Q X_{0}$ is just the conormal bundle of $Q$ in $X$. Similarly, define
    \[  \TXQcl = \pi\pb(\overline{Q}) \cap \TXK, \]
where $\cl{Q}$ is the closure of $Q$ in $X$. Let $\tl: \TXQcl \hookrightarrow \TX$ be the embedding.

Now consider the sheaf of commutative $\OO_X$-algebras 
  \[  \Sl := \Dr|_{t=0} = \Dr / t \Dr. \]
There is another sheaf of commutative $\OO_X$-algebras
  \[  \pi_* \OO_{\TX} \cong S (\LLieg / \LLien) \otimes_{S \Lieh} \C_{\lnc}. \]
From Definition \ref{defn:drt} of $\Dr$ we have

\begin{lemma}\label{lemma:TXSl}
There is a canonical isomorphism of sheaves of $\OO_X$-algebras between $\Sl$ and $\pi_* \OO_{\TX}$.
\end{lemma}

Let $\Mod^K(\TX)$ be the abelian category of sheaves of $K$-equivariant quasi-coherent $\OO_{\TX}$-modules and let $\Mod^K_{coh}(\TX)$ be its subcategory consisting of coherent $\OO_{\TX}$-modules. Similarly, let $\Mod(\Sl,K)$ be he abelian category of sheaves of $K$-equivariant $\Sl$-modules and let $\Mod_{coh}(\Sl,K)$ be its subcategory consisting of coherent $\Sl$-modules. Then Lemma \ref{lemma:TXSl} implies

\begin{corollary}\label{cor:TXSl}
  The direct image functor $\pi_*$ induces an equivalence between $\Mod^K(\TX)$ and $\Mod(\Sl,K)$. Moreover, it restricts to an equivalence between $\Mod^K_{coh}(\TX)$ and $\Mod_{coh}(\Sl,K)$.
\end{corollary}

We now study the geometry of $(\Dr,\KK)$-modules as in Definition \ref{defn:DrK}. Let $\W$ be any $(\Dr,\KK)$-module and $\W_0 = \W|_{t=0}$. Then $\W_0$ is a $(\Sl,K)$-module and therefore $\W_0 = \pi_* \widetilde{\W}_0$ for some $K$-equivariant $\OO_{\TX}$-module $\widetilde{\W}_0$ by Corollary \ref{cor:TXSl}. We now exhibit a constraint on possible supports of such $\widetilde{\W}_0$. First define the sheaf of commutative algebras over $\OO_X$,
   \[  \Spl:= \pi_* \OO_{\TXK}. \]
Note that the closed embedding $\TXK \hookrightarrow \TX$ induces a surjective homomorphism of sheaf of commutative algebras over $\OO_X$
  \[  \Sl \twoheadrightarrow \Spl.  \]

\begin{lemma}\label{lemma:spl}
  With the notations above, the $\Sl$-module structure of $\W_0$ factors through the morphism $\Sl \to \Spl$. In particular, the support of $\widetilde{\W}_0$ is a $K$-invariant closed subvariety of $\TXK$.
\end{lemma}

\begin{proof}
  It amounts to showing that the action of $\LLiek_t$ on $\W$ induced from the action of $\Dr$ vanishes after descending to $\W_0$. This is immediate due to the $K$-equivariant structure and its compatibility with the $\Dr$-module structure.
\end{proof}

By Corollary \ref{cor:TXSl}, we can regard $i_\sharp \phi |_{t=0}$ as a $K$-equivariant quasi-coherent sheaf $\wt{\EEE}=\wt{\EEE}(\lambda,Q,\phi)$ of $\OO_{\TX}$-modules. In other words, $\wt{\EEE}$ is the unique $\OO_{\TX}$-module (up to isomorphism) such that 
  \begin{equation}\label{eq:EIt}
    \pi_* \wt{\EEE} \simeq i_\sharp \phi \big|_{t=0} 
  \end{equation} 
as $(\Sl,K)$-modules. Denote the restriction of $\pi: \TX \to X$ to $\TXQ$ by $\pl: \TXQ \to Q$. Similarly the restriction of $\pi$ to $\TXQcl$ is denoted by $\bpl: \TXQcl \to \cl{Q}$. Set $\varphi := \omega_{Q/X} \otimes_{\OO_Q} \phi$. The following lemmas describe $\wt{\EEE}$.

\begin{lemma}\label{lemma:iSK}
  There is a canonical isomorphism of sheaves of $\OO_Q$-algebras
    \[  i^* \Spl \cong p_* \OO_{\TXQ}.  \]
\end{lemma}

\begin{proof}
  Consider the Cartesian diagram
  \begin{equation}\label{diag:TXQbase}
     \begin{diagram}
     \TXQ  & \rInto^{\ti}  & \TX  \\ 
      \dTo^{p} & & \dTo^{\pi} \\    
    Q & \rInto^{i}  & X.  \\
  \end{diagram}
  \end{equation}
  
  Since $\pi$ is affine, the lemma follows from base change.
\end{proof}

\begin{lemma}\label{lemma:spl_phi}
  There is a canonical isomorphism of $\OO_{i^* \Spl}$-modules,
    \[ \DQt \otimes_{\U \LLiek_t} \phi[t] \big|_{t=0} \cong (i^*\Spl) \otimes_{\OO_Q} \varphi.  \]
\end{lemma}
 
\begin{lemma}\label{lemma:wtE}
 $\wt{\EEE} = \ti_* \pl^* \varphi$.
\end{lemma}
\begin{proof}
  First of all, by Lemma \ref{lemma:spl_phi} and Lemma \ref{lemma:iSK} we have
    \[  \DQt \otimes_{\U \LLiek_t} \phi[t] |_{t=0} = (i^*\Spl) \otimes_{\OO_Q} \varphi = (p_* p^* {\OO_Q}) \otimes_{\OO_Q} \varphi = (p_* \OO_{\TXQ}) \otimes_{\OO_Q} \varphi = p_* p^* \varphi, \]
  where the last equality is by the projection formula.
  The commutative diagram \ref{diag:TXQbase} gives
    \[   i_\sharp \phi |_{t=0} = i_* p_* p^* \varphi = \pi_* \ti_* p^* \varphi. \]
  Therefore $\wt{\EEE} = \ti_* \pl^* \varphi$ by Corollary \ref{cor:TXSl}. 
\end{proof}

Let $\wt{\I}_0(\lambda,Q,\phi) = i_\flat \phi |_{t=0}$, which is a coherent $(\Sl,K)$-module. Therefore there is an unique $K$-equivariant coherent $\OO_{\TX}$-module $\EEE= \EEE(\lambda,Q,\phi)$ (up to isomorphism) such that 
  \begin{equation}\label{eq:EI}
     \pi_* \EEE \simeq  \wt{\I}_0(\lambda,Q,\phi). 
  \end{equation}  
The inclusion $i_\flat \phi \hookrightarrow i_\sharp \phi$ induces a natural morphism of sheaves of $\OO_{\TX}$-modules by restriction to $t=0$,
  \[ \eta: \EEE \to \wt{\EEE}.     \]
To describe the image of subsheaf of $\eta$, consider the following commutative diagram
\begin{equation}
  \begin{diagram}
    \TXQ  & \rInto^{\ti'} & \TXQcl  & \rInto^{\tl} &\TX  \\
    \dTo^{\pl}   &  & \dTo^{\bpl} & & \dTo^{\pi}  \\
    Q  & \rInto^{i'} & \cl{Q} & \rInto^{l} & X
  \end{diagram}
\end{equation}
where the horizontal maps satisfy $\ti = \tl \circ \ti'$, $i = l \circ i'$. Note that $\M(Q,\phi)$ is a subsheaf of $i_* \varphi$ by Theorem \ref{thm:minimal-K}. Since $\M(Q,\phi)$ is supported over $\cl{Q}$, we may regard $\M(Q,\phi)$ as a sheaf over $\cl{Q}$ by abuse of notation. Then $\M(Q,\phi)$ is a subsheaf of $i'_*  \varphi$ and hence also a subsheaf of $i'_* \pl_*\pl^*  \varphi = \bpl_* \ti'_* \pl^*  \varphi$ via the injective natural morphism 
  \[i'_* \varphi \hookrightarrow i'_* (\pl_* \OO_{\TXQ} \otimes_{\OO_{Q}} \varphi) =  i'_* \pl_*\pl^*  \varphi, \] 
where the equality is by the projection formula. Applying the adjunction of the pair $(\bpl^*,\bpl_*)$ to the resulting injective morphism
  \[  \MQ \hookrightarrow  \bpl_* \ti'_* \pl^*  \varphi \]
  gives us a morphism of $\OO_{\TXQcl}$-modules
  \begin{equation}\label{eq:tM}
     \bpl^* \MQ \to \ti'_*  \pl^*  \varphi, 
   \end{equation}
which however might not be an injective morphism since $\bpl$ is in general not flat. Denote the image subsheaf of this morphism by $\tMQ \subset  \ti'_*  \pl^*  \varphi$. Apply the functor $l_* \bpl_*$ to the morphism \eqref{eq:tM} to get a morphism of $\Sl$-modules,
  \begin{equation}\label{eq:multimap}
     l_* \bpl_* \bpl^* \MQ \to l_* \bpl_* \ti'_*  \pl^*  \varphi = \pi_* \tl_* \ti'_*  \pl^*  \varphi = \pi_* \ti_*  \pl^*  \varphi = \pi_* \wt{\EEE} = i_\sharp \phi |_{t=0}, 
   \end{equation}
where the second last equality is by Lemma \ref{lemma:wtE}. Since $l$ is a closed embedding and $\bpl$ is affine, both $l_*$ and $\bpl_*$ are exact and hence the image subsheaf of the morphism above is 
  \[ l_* \bpl_* \tMQ = \pi_* \tl_* \tMQ. \] 
% Combining this with \eqref{eq:EI} and Corollary \ref{cor:TXSl}, we have proven

\begin{proposition}\label{prop:E}
  The image subsheaf of $\eta$ is $\tl_* \tMQ$.
\end{proposition}
\begin{proof}
  We have
    \[    l_* \bpl_* \bpl^* \MQ \cong \Spl \otimes_{\OO_X} \MQ \]
  by the projection formula $\bpl_* \bpl^* \MQ \cong \bpl_* \OO_{\TXQcl} \otimes_{\OO_{\cl{Q}}} \MQ$, hence the morphism\eqref{eq:multimap} reads as
  \[ \Spl \otimes_{\OO_X} \MQ \to i_\sharp\phi |_{t=0}. \] 
  By Proposition \ref{lemma:spl}, it is induced by the multiplication map
    \[ \Sl \otimes_{\OO_X} \MQ \to i_\sharp\phi |_{t=0} \]
  and therefore its image subsheaf $\pi_* \tl_* \tMQ$ is equal to the $\Sl$-submodules of $i_\sharp\phi|_{t=0}$ generated by $\MQ$, which is exactly $\pi_* \Image(\eta)$.
\end{proof}

From the discussion above, we have
\begin{lemma}
 The supports of $\wt{\EEE}$ and $\EEE$ are in $\TXQcl \subset \TX$.
\end{lemma}

%\begin{lemma}
 %The abelian category $\Mod(\Sl,K)$ of quasi-coherent $K$-equivariant $\Sl$-modules over $X$ is equivalent to the abelian category $\Mod_K(\TXK)$ of quasi-coherent $K$-equivariant sheaves over $\TXK$.
%\end{lemma}

It is therefore natural to give the following definition.

\begin{definition}
  The \emph{($\lnc$-)twisted characteristic variety} of the $\Dl$-module $\ilq$ is the subvariety
    \[ Ch_{\lnc}(\ilq) := \supp(\EEE) \subset \TXQcl \subset \TX. \]
\end{definition}

When $\lnc$ is generic, $Ch_{\lnc}(\ilq)$ is smooth, but in general it is singular. In the most singular case when $\lnc=0$, we have $\lt = \lc$. Thus $\Dr$ is the Rees algebra of $\Dl$ with respect to its natural filtration, whereas $i_\flat \phi$ is the Rees module of $i_+\phi$ with respect to the filtration generated by $\M(Q,\phi)$. Thus we have

\begin{lemma}
  When $\lnc=0$, $Ch_{0}(\ilq)$ is the usual characteristic variety of the $\Dl$-module $\ilq$ in the usual cotangent bundle $T^*X$.
\end{lemma}

\begin{remark}
  Similar to the usual cotangent bundle, the twisted cotangent bundle $\TX$ is equipped with a natural $G$-invariant algebraic symplectic form. The smooth locus of the twisted characteristic varieties are in fact Lagrangian subvarieties in $\TX$. This will be proven elsewhere.
\end{remark}

Pick a point $x \in Q$ and a a $\theta$-stable Cartan subalgebra $\Lieh_x \in \Lieb_x$. So $\Lien_x^\bot = \Lieh^*_x \oplus \Lieb_x^\bot$ and we can regard $\lnc$  with a semisimple element $\lncp$ in $\Lieh_x^* \cap \Lies^*$. Let $O_\lnc$ be the closed $K$-orbit of $\lncp$ in $\Lies^*$. It only depends on $\lnc$ and is independent of the choice of $x$ or $\Lieh_x$ (hence of $\lncp$). Set
  \[  \Ql := \mu\pb(O_\lnc) \cap \TXQ, \]
which is a $K$-invariant subvariety of $\TXQ \subset \TXK$ and only depends on the choice of $Q$. On the other hand, we have a point $\tlnc = (\lncp, x)$ in $\mu\pb(\lncp) \cap \TXQ$.

\begin{proposition}\label{prop:Ql}
  We have $\Ql = K \cdot \tlnc$. In other words, $\Ql$ is a single $K$-orbit in $\TXQ$.
\end{proposition}

For the proof, we need the following

\begin{lemma}\label{lemma:cartan_lnc}
  Suppose $\Lieb$ is a Borel subalgebra and $\gamma$ is a semisimple element in $\Lieb$ with $\theta \gamma = -\gamma$. Then there exists a $\theta$-stable Cartan subalgebra $\Liec$ in $\Lieb$ such that $\gamma \in \Liec$.
\end{lemma}

\begin{proof}
  Let $\Liel^\gamma$ be the centralizer of $\gamma$ in $\Lieg$, then $\Liel^\gamma$ is a $\theta$-stable reductive subalgebra of $\Lieg$, which contains $\gamma$. Moreover, $\Lieb^\gamma = \Liel^\gamma \cap \Lieb$ is a Borel subalgebra of $\Liel^\gamma$. By Lemma \ref{lemma:KCartan}, $\Liel^\gamma$ contains a $\theta$-stable Cartan subalgebra $\Liec$. Then $\Liec$ must contain $\gamma$ and it is also a Cartan subalgebra of $\Lieb$.
\end{proof}

\begin{proof}[Proof of Proposition \ref{prop:Ql}]
  Suppose $\wt{\gamma}=(\gamma,y) \in \Ql$. We need to prove that $\wt{\gamma}$ and $\tlnc$ are $K$-conjugate. Since both $x$ and $y$ are in $Q$, we can conjugate $\wt{\gamma}$ by some element in $K$ if necessary and assume $x=y$ and $\gamma \in \Lieb_x$. By Lemma \ref{lemma:cartan_lnc}, there is a $\theta$-stable Cartan subalgebra $\Liec$ in $\Lieb_x$ which contains $\gamma$. Similarly there is another $\theta$-stable Cartan subalgebra $\Liec'$ in $\Lieb_x$ containing $\lncp$. By Lemma \ref{lemma:KCartan}, there exists $k \in K \cap N_x$ such that $\Liec = \Ad(k)\Liec'$. Now both $\Ad(k) \lncp$ and $\gamma$ belong to $\Liec$ and their images under the specialization $\Liec \cong \Liec^* \cong (\Lieb_x/\Lien_x)^* \cong \Lieh^*$ are $\lnc$. Hence $\gamma= \Ad(k) \lncp$ and $\wt{\gamma} = \Ad(k) \tlnc$.
\end{proof}

%Fix a point $x \in Q$ and a $\theta$-stable Cartan subalgebra $\Lieh_x \in \Lieb_x$.  a point, which determines a point $\tlnc = (\lnc, x) \in \mu\pb(\lnc) \subset \TX$. Denote by $\Ql$ the $K$-orbit of $\lnc$ considered as a point in $\TXQ \subset \TX$ via the identification \eqref{eq:b_ident}. Its image under $\mu$ is the closed $K$-orbit of $\lnc$ as a point in $\Lies^*$, denoted by $O_\lnc$ ($\lnc$ is semisimple). Note that, while $O_\lnc$ only depends on $\lnc$, $Q_\lnc$ also depends on the choice of $Q$. 

Denote by $\Qcl$ the closure of $\Ql$ in $\TX$. We have the following commutative diagram,

\begin{equation}\label{diag:GS}
  \begin{diagram}
    Q                     & \rTo^{\Id}  & Q & \rInto^{i}  & X  \\
    \uTo^{p}   & &  \uTo^{p} & & \uTo^{\pi} \\
    \Ql & \rInto  & \TXQ  & \rInto^{\ti}  & \TX  \\
    \dTo & &\dTo  & & \dTo^{\mu}  \\
    O_\lnc & \rInto & \Lies^* & \rInto & \Lieg^*   
  \end{diagram}
\end{equation}

\begin{comment}
 Denote the restriction of the vector bundle $\wt{\EEE}$ to $\Ql$ by $\varphi$. The restriction $p|_{\Ql} : \Ql \to Q$ is a fibration and we have
  \[  \varphi \cong (p|_{\Ql})^* ( \omega_{Q/X} \otimes_{\OO_Q} \phi), \]
by the diagram \eqref{diag:GS}. Therefore the induced homomorphism
  \[  \Gamma(Q, \omega_{Q/X} \otimes_{\OO_Q} \phi) \to \Gamma(\Ql, \varphi)  \]
is injective. Since all the minimal $K$-types of $\Gamma(X, \ilq)$ appear in $\Gamma(Q, \omega_{Q/X} \otimes_{\OO_Q} \phi)$ by Theorem \ref{thm:minimal-K}, we have therefore shown

\begin{proposition}
  All the minimal $K$-types of $\Gamma(X, \ilq)$ appear in $\Gamma(\Ql, \varphi)$. 
\end{proposition}

\begin{corollary}
  All the minimal $K$-types of $\Gamma(X, \ilq)$ appear in $\Gamma(\Qcl, \EEE)$.
\end{corollary}

Since $\M(Q,\phi)$ is a subsheaf of $i_* \varphi$ by Theorem \ref{thm:minimal-K}, the sheaf $\pi^* \M(Q,\phi)$ is a subsheaf of $\wt{\EEE}$ by Lemma \ref{lemma:wtE}.

\begin{lemma}
  $\EEE = p^* \M(Q,\phi)$.
\end{lemma}

Set
  \[  \V = \VQ := \iota_* \iota^* \EEE \in \cohK(\TX). \]
Then the natural morphism $\EEE \to \V$ is surjective and so is the induced morphism
  \[  \pi_* \EEE \to \pi_* \V \]
in $\coh(\Sl,K)$ since $\pi$ is affine. If we set
  \begin{equation}
    \iqzero : = \pi_* \V =  \pi_* \VQ,
  \end{equation}
then by \eqref{eq:EI} we get a surjective morphism
  \[ \iqt|_{t=0} \to \iqzero.  \]

\begin{lemma}
  $\iqzero$ is the unique irreducible quotient of $\iqt|_{t=0}$ as a $(\Sl,K)$-module.
\end{lemma}

\end{comment}

Now assume $\lnc$ is purely imaginary and set
  \[  M(\lambda,Q,\phi) := \Gamma(X,\ilq)  \] 
is an irreducible tempered $(\Lieg,K)$-module. Let $P^\lnc=L^\lnc N^\lnc$ be the parabolic subgroup of $G$ with Lie algebra $\Liep^\lnc$ containing $\Lieb_x$, whose Levi factor $L^\lnc$ is the centralizer of $\lnc$. Then $P^\lnc = \overline{P^\lnc}$, $L^\lnc = \overline{L^\lnc}$, and $L^\lnc$ is $\theta$-stable and all imaginary roots are contained in roots of $\Liel^\lnc$, which implies that $K \cap P^\lnc = K \cap L^\lnc$. Let $\Llr = L^\lnc \cap \GR$ be the centralizer of $\lnc$ as an element in $i \Lieg^*_\R$ under the coadjoint $\GR$-action, then $L^\lnc$ is the complexification of $\Llr$ and $K \cap L^\lnc = K^\lnc$ is the complexification of the centralizer $\KR \cap \Llr = K^\lnc_\R$ of $\lnc$ in $\KR$. Moreover, $(\Llr, K^\lnc_\R)$ is a real reductive pair. Let $S$ be the set of all simple roots vanishing on $\lnc$. Consider the associated fibration $\pi_S: X \to X_S \cong G/P^\lnc$ of the full flag variety over the partial flag variety of type $S$. Set $y = \pi_S(x)$, then $X_y=\pi\pb_S(y) \cong P^\lnc/B \cong L^\lnc / L^\lnc \cap B$ is the flag variety for the Harish-Chandra pair $(\Liel^\lnc,K^\lnc)$, where $\Liel^\lnc$ is the Lie algebra of $L^\lnc$. Consider the data $( \lambda, Q_L=K^\lnc \cdot x, \phi_L)$ on $X_y$ for the pair $(\Liel^\lnc,K^\lnc)$, where $\phi_L$ is the $K^\lnc$-homogeneous connection on the $K^\lnc$-orbit $Q_L$ induced by the geometric fiber $\tau$ of $\phi$ at $x$ regarded as a $K^\lnc_x$-representation ($K^\lnc_x = K^\lnc \cap K_x$). Hence $\Gamma(X_y, \I( \lambda, Q_L, \phi_L))$ is a $(\Liel^\lnc,K^\lnc)$-module. Denote by $j$ the embedding of $Q_S$ into $X_S$ and by $k$ the embedding of $X_y$ into $X$. We have the following the diagram \ref{diag:Q_L},

\begin{equation}\label{diag:Q_L}
\begin{diagram}
       Q_L    & \rInto  & X_y   \\
       \dInto  &           &\dInto^{k} \\
       Q        & \rInto^{i}        & X   \\
       \dTo       &  & \dTo_{\pi_S}  \\
       Q_S        &\rInto^{j}         & X_S
\end{diagram}
\end{equation}

Recall in \S 4, \cite{Chang}, a $K$-equivariant sheaf $\tD_\lambda$ of algebras was defined over $X_S$, which is an analogue of $\D_\lambda$ over $X$. Moreover, $(\pi_S)_*$ induces a functor
  \[  (\pi_S)_* \Mod(\D_\lambda,K) \to \Mod(\tD_\lambda,K).  \]
Another sheaf $\tD^j_\lambda$ of algebras was defined over $Q_S$, which is an analogue of $\D^i_\lambda$ over $Q$. One can define a functor
  \[  j_+: \M(\tD^j_\lambda, K) \to \M(\tD_\lambda,K)   \] 
in a similar way as how $i_+$ is defined. There is an exact induction functor (Lemma 5.2., \cite{Chang})
\[
 \ind:  \Mod(\Liel^\lnc,K^\lnc) \to \Mod(\widetilde{\D}^j_\lambda,K). 
\]
The following proposition essentially says that the $(\Lieg,K)$-module $\Gamma(X,i_+\phi)$ is cohomologically induced from the $(\Liel^\lnc,K^\lnc)$-module $\Gamma(X_y, \I( \lambda, Q_L, \phi_L))$ (cf., \cite{KnappVogan}, \cite{Oshima1}, \cite{Oshima2}).

\begin{proposition}\label{prop:ind}
  There is a natural isomorphism in $\Mod(\D_\lambda,K)$,
    \[ (\pi_S)_* i_+ \phi  \cong j_+(\ind(\Gamma(X_y, \I( \lambda, Q_L, \phi_L)))). \]
\end{proposition} 
\begin{proof}
  This follows immediately from Theorem 5.4 and Theorem 4.14 of \cite{Chang} and it generalizes Corollary 5.8 of \emph{loc. cit.} in which $S$ is consists of a single simple root.
\end{proof}

On the other hand, suppose $L^\lnc_\R = M^\lnc_\R A^\lnc_\R$ is the Langlands decomposition with its complexification $L^\lnc=M^\lnc A^\lnc$, then $V_L := \Gamma(X_y, \I( \lambda, Q_L, \phi_L))$ is also an $(\Liem^\lnc, K^\lnc)$-module. 

\begin{lemma}\label{lemma:repn_M}
  $V_L$ is an irreducible tempered $(\Liem^\lnc, K^\lnc)$-module with real infinitesimal character.
\end{lemma}  
  
\begin{proof}
  Since $\Gamma(X,\ilq)$ is irreducible,  $\Gamma(X_y, \I( \lambda, Q_L, \phi_L))$ must be irreducible by Proposition \ref{prop:ind}. To show it has real infinitesimal character, fix a $\theta$-stable Cartan subalgebra $\Lieh_\R$ of $L^\lnc_\R$ and its complexification $\Lieh$. Then there is decomposition
  \[  \Lieh = (\Lieh \cap \Liem^\lnc) \oplus \Liea^\lnc = \Lieh^\lnc \oplus \Liea^\lnc, \] 
where $\Lieh^\lnc = \Lieh \cap \Liem^\lnc$. By Lemma 3.4 (4), \cite{VoganLanglands}, the restriction of $\lnc$ to $\Lieh^\lnc$ is zero and so the infinitesimal character of $\Gamma(X_y, \I( \lambda, Q_L, \phi_L))$ as an $\Liem^\lnc$-module is $\lc$.
\end{proof}

\begin{corollary}\label{cor:repn_M}
  The $(\Lieg,K)$-module $\ilq$ is (the Harish-Chandra module of) the real parabolic induction of the $L^\lnc_\R$-representation (or $(\Liel^\lnc,K^\lnc)$-module) $V_L$.
\end{corollary}
\begin{proof}
  The conclusion can be deduced from Proposition \ref{prop:ind} above, the identification of parabolic induction and cohomological induction (Proposition 11.47 of \cite{KnappVogan}) and the identification of $\D$-modules with cohomologically induced modules (\cite{Oshima1}, \cite{Oshima2}).
\end{proof}

Denote by $\M_{L}(Q_L,\phi_L)$ the $\OO_{X_y}$-module generated by minimal $K^\lnc$-types of the $(\Liel^\lnc,K^\lnc)$-module $V_L$. By Lemma \ref{lemma:repn_M} and Theorem \ref{thm:Vogan}, $V_L$ has a unique minimal $K^\lnc$-type, denoted by $\sigma$. Note that $\delta=(\lnc,\sigma)$ defines a Mackey datum and recall we have the $\GR$-representation $M(\delta)$ and the $M_\R^\lnc$-representation $V_{M_\R^\lnc}(\delta)$. Then Theorem \ref{thm:Vogan} and Corollary \ref{cor:repn_M} immediately implies
\begin{lemma}\label{lemma:MHA_D}
 $V_L$ is the Harish-Chandra module of $V_{M_\R^\lnc}(\delta)$ and $M(\lambda,Q,\phi) = \Gamma(X,\ilq)$ is the Harish-Chandra module of $M(\delta)$. 
\end{lemma} 

\begin{proposition}\label{prop:resM}
  The restriction of $\M(Q,\phi)$ to $X_y$ is isomorphic to $\M_{L}(Q_L,\phi_L)$ as $\OO_{X_y}$-module.
\end{proposition}

\begin{proof}
  It is immediate from our Definition \ref{defn:special_revised} of special $K$-types that there is a natural injective morphism $r: k^*\M(Q,\phi) \to \M_{L}(Q_L,\phi_L)$, where $k: X_y \to X$ is the embedding. Since $\Gamma(X_y, \I( \lambda, Q_L, \phi_L))$ is an irreducible tempered $(\Liem^\lnc, K^\lnc)$-module has the unique minimal $K^\lnc$-type $\sigma$. Thus the morphism $r$ must be surjective and hence an isomorphism.
\end{proof}

Since the stabilizer of $\lnc \in \Lieg^*$ in $G$ is $L^\lnc$, $L^\lnc$ acts on $\mu\pb(\lnc)$. The map $\mu|_{\Ql}: \Ql \to O_\lnc$ is a $K$-equivariant fibration. Denote its fiber over $\lnc \in O_\lnc$ by $Z_\lnc : =  \mu\pb(\lnc) \cap \Ql$. The group $K^\lnc$ acts on $Z_\lnc$ and its closure $\cl{Z}_\lnc = \mu\pb(\lnc) \cap \Qcl$. Denote by $K^\tlnc$ the $K$-stabilizer of $\tlnc \in \TX$. Then $K^\tlnc = K^\lnc_x=K^\lnc \cap K_x$.

\begin{lemma}\label{lemma:fiber}
   The restriction of $\pi: \TX \to X$ induces a $K^\lnc$-equivariant isomorphism between $Z_\lnc$ and $Q_L$. Similarly, it also induces a $K^\lnc$-equivariant isomorphism between $\cl{Z}_\lnc$ and $\cl{Q}_L$.
\end{lemma}

\begin{proof}
  Let $Y_\lnc := G \cdot \tlnc \cap \mu\pb(\lnc) \subset \TX $. The stabilizer of $\lnc$ in $G$ is $L^\lnc$ and it acts transitively on $Y_\lnc$. The stabilizer of $\tlnc=(\lnc, x)$ in $G$ is $L^\lnc \cap B_x$, where $B_x$ is the Borel subgroup associated to $x \in X$. Hence $Y_\lnc \cong L^\lnc / L^\lnc \cap B_x$ and therefore the restriction of $\pi: \TX \to X$ induces a $L^\lnc$-equivariant isomorphism between $Y_\lnc$ and $X_y$. This isomorphism restricted to $K^\lnc$-equivariant isomorphisms between $Z_\lnc$ and $Q_L$ and between $\cl{Z}_\lnc$ and $\cl{Q}_L$.  
\end{proof}

We have the following diagrams whose top lines are fibrations,
\begin{diagram}
   Z_\lnc  & \rInto & \Ql  & \rOnto^{\mu}  & O_\lnc   &  &  &  &  \cl{Z}_\lnc  & \rInto & \Qcl  & \rOnto^{\mu}  & O_\lnc.   \\
   \dTo^\cong &  & \dOnto &  &  &  &  &  & \dTo^\cong &  & \dTo  \\
   Q_L  & \rInto  & Q & &  &  &  &  &  \overline{Q}_L  & \rInto  & \overline{Q}    
\end{diagram}

\begin{corollary}\label{cor:fiber}
  The map $\mu|_{\Qcl}: \Qcl \to O_\lnc$ is a K-equivariant fibration whose fiber over $\lnc \in O_\lnc$ is isomorphic to $\overline{Q}_L$ via the projection $\pi: \TX \to X$.
\end{corollary}

\begin{proposition}
  The $\OO_{\cl{Z}_\lnc}$-module $\cl{\kappa}^*\tMQ$ can be identified with $\M_{L}(Q_L,\phi_L)$ via the isomorphism $\cl{Z}_\lnc \cong \cl{Q}_L$ in Lemma \ref{lemma:fiber}.
\end{proposition}
\begin{proof}
  First consider the Cartesian diagram
  \begin{diagram}
    Z_\lnc   & \rInto^{\iota}  & \cl{Z}_\lnc   \\
    \dInto^{\kappa}    &     & \dInto^{\cl{\kappa}}   \\
    \TXQ    & \rInto^{\wt{i}'}   & \TXQcl
  \end{diagram}
  Note that horizontal maps are open embeddings and vertical maps are closed embeddings. Therefore we have isomorphism of $\OO_{Z_\lnc}$-modules
    \begin{equation}\label{isom:Zlnc}
      \cl{\kappa}^* \wt{i}'_* p^* \varphi \cong \iota_* \kappa^* p^* \varphi. 
    \end{equation}
  By Lemma \ref{lemma:fiber}, we have the following diagrams
  \begin{diagram}
   Z_\lnc  & \rInto^{\kappa} & \TXQ  & & & & \cl{Z}_\lnc  & \rInto^{\cl{\kappa}} & \TXQcl    \\
   \dTo^\cong &  & \dTo^{\pl} & & & & \dTo^\cong &  & \dTo^{\bpl}  \\
   Q_L  & \rInto  & Q  & & & &  \overline{Q}_L  & \rInto  & \cl{Q}
\end{diagram}
 which imply that $\iota_* \kappa^* p^* \varphi$ can be identified with $(\iota_L)_* (\varphi|_{Q_L})$ via $\cl{Z}_\lnc \cong \cl{Q}_L$, where $\iota_L : Q_L \hookrightarrow \cl{Q}_L$ denotes the open embedding. Similarly $\cl{\kappa}^* \bpl^* \MQ$ can be identified with the restriction of $\M(Q,\phi)$ to $\cl{Q}_L$ via $\cl{Z}_\lnc \cong \cl{Q}_L$, which is isomorphic to $\M_L(Q_L,\phi_L)$ by Proposition \ref{prop:resM}. The natural morphism of $\OO_{\cl{Z}_\lnc}$-modules
   \[  \cl{\kappa}^* \bpl^* \MQ \to \iota_* \kappa^* p^* \varphi  \]
is induced by the morphism \eqref{eq:tM} and \eqref{isom:Zlnc}, hence can be identified with the natural injective morphism of $\OO_{\cl{Q}_L}$-modules (here we regard $\M_L(Q_L,\phi_L)$ as its restriction over $\cl{Q}_L \subset X_y$),
  \[ \M_L(Q_L,\phi_L)  \hookrightarrow (\iota_L)_* (\varphi|_{Q_L}). \]
  Since $\cl{\kappa}$ is a closed embedding, $\cl{\kappa}^*$ is exact and preserves image subsheaves. Hence we have proven that $\cl{\kappa}^* \tMQ$ can be identified with $\M_L(Q_L,\phi_L)$ via $\cl{Z}_\lnc \cong \cl{Q}_L$. Finally the statement follows from Proposition \ref{prop:E}.
\end{proof}

Define the $K$-equivariant vector bundle $\EEE_\sigma := K \times_{K^\lnc} V_\sigma$ over $O_\lnc \cong K / K_\lnc$. Denote by $q: \Qcl \hookrightarrow \TX$ the embedding.  Let $\V = q_*q^* \tl_* \tMQ$, then it is an irreducible $K$-equivariant coherent sheaf over $\TX$ and hence $\pi_* \V$ is a coherent $(\Sl,K)$-module. Moreover, $\V$ is the unique maximal quotient of $\EEE$ in the category $\Mod^K_{coh}(\TX)$ via the composition
  \[  \EEE \xrightarrow{\eta} \tl_* \tMQ \to \V.  \] 

\begin{proposition}
  The sheaf $\mu_*\V$ is isomorphic to $\EEE_\sigma$ over $O_\lnc$.
\end{proposition}

\begin{proof}
  By Lemma \ref{lemma:fiber} and Proposition \ref{prop:resM}, we only need to show that $\Gamma(X_y, \M_{L}(Q_L,\phi_L)) = V_\sigma$. This follows from Proposition \ref{prop:secM} applied to the group $L$ and its flag variety $X_y$.
\end{proof}

Set 
  \[ \I_0(\lambda,Q,\phi) := \pi_* \V, \]
then it is an irreducible quotient of $\pi_* \EEE = \wt{\I}_0 (\lambda,Q,\phi)$ as an $(\Sl,K)$-module. Set 
  \[ \wt{M}_0(\lambda,Q,\phi) := \Gamma(X, \wt{\I}_0 (\lambda,Q,\phi)), \quad  M_0(\lambda,Q,\phi) := \Gamma(X, \I_0 (\lambda,Q,\phi)).  \]
Then they are $(\Lieg_0,K)$-modules and $M_0(\lambda,Q,\phi)$ is a quotient of $\wt{M}_0(\lambda,Q,\phi)$. Moreover, $M_0(\lambda,Q,\phi) = \Gamma(O_\lnc, \EEE_\sigma)$ is exactly the irreducible $(\Lieg_0,K)$-module of the $G_{\R,0}$-representation $M_0(\delta)$ defined in \eqref{eq:M0}. The $(\Lieg,K)$-module $M(\lambda,Q,\phi)$ and the $(\Lieg_0,K)$-module $\wt{M}_0(\lambda,Q,\phi)$ are connected by the $\gkt$-module
  \[ M(\lt, Q, \phi) := \Gamma(X, \iqt).  \]
Together with Corollary \ref{lemma:MHA_D}, we have proved the following main result claimed in the Introduction.

\begin{theorem}\label{thm:main}  
  Any irreducible tempered $(\Lieg,K)$-module of the form $M(\lambda,Q,\phi)$ can be extended to a $\gkt$-module $M(\lt, Q, \phi)$, such that its restriction $\wt{M}_0(\lambda,Q,\phi)$ at $t=0$ admits an irreducible quotient $M_0(\lambda,Q,\phi)$ as $(\Lieg_0,K)$-module. Moreover, the correspondence
    \[  M(\lambda,Q,\phi) \longleftrightarrow M_0(\lambda,Q,\phi) \]
coincides with the MHA bijection. In particular, the correspondence is independent of the realization of $M(\lambda,Q,\phi)$ as global sections of standard Harish-Chandra sheaf (i.e., choice of the data $(\lambda,Q,\phi)$).
\end{theorem}

\begin{remark}
 % In fact the condition that $\lambda$ in Theorem \ref{thm:Dtempered} can be weakened.
  Lemma 4.1.4 of \cite{Mirkovic} says that a tempered representation can be realized as global sections of standard Harish-Chandra sheaves associated to different data $(\lambda,Q,\phi)$ with the same $K$-conjugacy class of Cartan subalgebra. These Harish-Chandra sheaves are related by intertwining functors. Moreover, it can even happen that the same representation is attached to different triple $(\lambda,Q,\phi)$ with different $K$-conjugacy class of Cartan subalgebras. In this case they are related by Schmid identities. See \cite{Milicic} for the example of a spherical principal series representation of $\mathrm{SU}(2,1)$ realized on the open $K$-orbit, which coincides with a limit of discrete series representation realized on a closed $K$-orbit. On the other hand, Theorem \ref{thm:main} and the MHA bijection is independent of these choices. A natural question is that whether $M(\lt,Q,\phi)$ and hence $\wt{M}_0(\lambda,Q,\phi)$ are independent of the choice of $(\lambda,Q,\phi)$. We expect that the sheaves $\iqt$ for different data $(\lambda,Q,\phi)$ are related by family version of intertwining functors or Schmid identities, hence the definition of the $\gkt$-module $M(\lt, Q, \phi)$ is intrinsic. 
 %We will reinterpret it in the framework of deformation quantization in future work.
\end{remark}

\begin{remark}
  The tempered assumption is not necessary for our approach, so the MHA bijection can be easily generalized to all admissible representations in our framework. This has already been pointed out by Afgoustidis to the author via private communication. The main difference in the admissible case is that the standard sheaf $\ilq$ might be reducible and the image of the canonical morphism
  \begin{equation}\label{eq:mor_iqt}
     \iqt|_{t=1} \to \tiqt|_{t=1} = \ilq 
  \end{equation}
is the unique irreducible Harish-Chandra subsheaf $\llq$ of $\ilq$, which is generated by the minimal $K$-types (we assume $\lambda$ is dominant). This phenomenon has already showed up in the tempered case: there are countably many complex values $s$ of $t$ such that $\I(\lambda_s,Q,\phi)$ is reducible, even when $\lnc$ is imaginary. 

 On the other hand, besides the `maximal extension' $\ilq=i_+ \phi$, there is also a `minimal extension' $i_! \phi$, which is also a $\Dl$-module, and there is a canonical morphism $i_! \phi \to i_+ \phi$ whose image is the `intermediate extension' $\llq$. This leads to the following conjecture.
\end{remark}

\begin{conjecture}
  There is a canonical isomorphism of $(\Dl,K)$-modules $i_! \phi \cong \iqt|_{t=1}$, which identifies the canonical morphism \ref{eq:mor_iqt} with the canonical morphism $i_! \phi \to i_+ \phi$.
\end{conjecture}

\subsection{Examples}\label{subsec:examples}
In the special case when the $K$-orbit $Q$ is closed and $\lambda$ is dominant, the minimal $K$-type of the standard module $\ilq$ is exactly $\Gamma(Q,\phi \otimes_{\OO_Q} \omega_{Q/X})$. Thus $M(Q,\phi) = i_* ( \phi \otimes_{\OO_Q} \omega_{Q/X})$. In this case, $i_\flat \phi = i_\sharp \phi$ and the corresponding $\OO_{\TX}$-module is $\EEE = \wt{\EEE} = \ti_* \pl^* (\phi \otimes_{\OO_Q} \omega_{Q/X})$ (Lemma \ref{lemma:wtE}). The first example is when the Cartan algebra associated to $Q$ is compact. Then the standard module $\ilq$ is a (limit of) discrete series representation. Since $\lnc = 0$, $\Ql$ is identified with $Q$ and $\V = \phi \otimes_{\OO_Q} \omega_{Q/X}$. Then $M_0(\lambda,Q,\phi)$ is the minimal $K$-type $\Gamma(Q,\phi \otimes_{\OO_Q} \omega_{Q/X})$.

The another important case is when $\GR$ is a complex reductive group regarded as a real group. Fix a maxima compact subgroup $\KR \subset \GR$ with Lie algebra $\Liek_\R$, which is also a compact form of $\GR$ as a complex group. The complexified Lie algebra $\Lieg = \Lieg_\R \otimes_\R \C$ can be identified with the complex Lie algebra $\Lieg_\R \oplus \Lieg_\R$ under the complex Lie algebra isomorphism $\beta: \Lieg \to \Lieg_\R \oplus \Lieg_\R$ defined by $\beta(Y) = (Y,\overline{Y})$ for any $Y \in \Lieg_\R$, where the bar denotes the complex conjugate linear automorphism of $\Lieg_\R$ determined by the compact form $\Liek_\R$. The image of $\Liek=\Liek_\R \otimes_\R \C$ under $\beta$ is the diagonal subalgebra $\Delta(\Lieg_\R)$ of $\Lieg_\R \oplus \Lieg_\R$. Thus we can replace the pair $(\Lieg,\Liek)$ by the isomorphic $(\Lieg_\R \oplus \Lieg_\R, \Delta(\Lieg_\R))$. 

In this case, the associated flag variety is the product $\wt{X} = X \times X$ where $X$ is the flag variety of $\GR$ as a complex group and the $K$-action coincides with the diagonal action of $\Delta(\GR) \subset \GR \times \GR$ whose Lie algebra is $\Delta(\Lieg_\R)$. All $\Delta(\GR)$-orbits of $\wt{X}$ are related by Weyl group elements and there is only one $\Delta(\GR)$-conjugacy class of $\theta$-stable Cartan subalgebras of $\GR \oplus \GR$. This conjugacy class is represented by $\Lieh \oplus \Lieh$ with $\Lieh$ being the Cartan subalgebra of $\Lieg_\R$. The toroidal part $\Liet$ of $\Lieh \oplus \Lieh$ is the diagonal $\Lieh_\Delta$ of $\Lieh \oplus \Lieh$, while the split part 
  \[  \Liea = \Lieh_{-\Delta} := \{ (Y, -Y) \in \Lieh \oplus \Lieh | Y \in \Lieh  \}, \] 
which is also identified with $\Lieh$ by $(Y,-Y) \mapsto Y$. Now Lemma 4.1.4 of \cite{Mirkovic} applied in this case implies that any tempered representation of $\GR$ can be realized as global sections of a standard Harish-Chandra sheaf $\I(\lambda,\Delta(X),\phi)$ associated to the unique closed $\Delta(\GR)$-orbit $\Delta(X)$, the diagonal of $\wt{X}=X \times X$. Here $\lambda$ is $\Sigma^+$-positive, i.e., $\langle \lambda, \alpha^\vee \rangle \geq 0$ for any $\alpha \in \Sigma^+$, where $\Sigma^+$ is the positive root system associated to $\Lieg_\R \oplus \Lieg_\R$ and $\alpha^\vee$ is the dual root of $\alpha$. By the discussion at the beginning of this section, the corresponding $(\Dr,\KK)$-module is $\iqt = i_\flat \phi = i_\sharp \phi$. For $Q=\Delta(X)$, we regard $\lnc \in \Lieh_{-\Delta}$ as an element in $\Lieh$, $T^*_{Q}(\wt{X})_{\lnc}$ is canonically identified with $\TX$ and the natural projection $p: T^*_{Q}(\wt{X})_{\lnc} \to Q$ is identified with the projection $\pi: \TX \to X$. The corresponding sheaf $\EEE$ is the vector bundle $\pi^*(\phi \otimes_{\OO_Q} \omega_{Q/X}) = \pi^*(\phi \otimes_{\OO_X} \omega^{-1}_X)$. 
%It is not hard to see that the $\gkt$-module $ M(\lt, Q, \phi) = \Gamma(X, \iqt) $ is admissible in the sense of Definition \ref{defn:gktmod_adm}.

\vskip 4em

%\newpage
\bibliographystyle{alpha}
\bibliography{MackeySL2.bib}

\begin{comment}

\end{comment}

\end{document}